\documentclass{amsart}
\usepackage{amsmath}
\usepackage{amsfonts}
\usepackage{amssymb}
\usepackage{graphicx}%

\newtheorem{myrem}{Remark}[section]
\newtheorem{mytheo}{Theorem}[section]
\newtheorem{mylem}{Lemma}[section]

\newtheorem{mycor}{Corollary}[section]
\newtheorem{mypro}{Proposition}[section]
\newtheorem*{mycorollary1.2}{Corollary 1.2*}

\usepackage{fancyhdr}

\usepackage{color}

\theoremstyle{remark}

\numberwithin{equation}{section}

 \allowdisplaybreaks[4]
\begin{document}

\title{On Log-Concave-Tailed Chaoses  and the Restricted
Isometry Property}

\pagestyle{fancy}

%清除原页眉页脚样式
\fancyhf {} 

%R：页面右边；O：奇数页；\leftmark：表示“一级标题”

%C：页面中间
\fancyhead[CO]{\footnotesize  SUPREMA OF CHAOSES AND THE R.I.P.}

\fancyhead [CE]{\footnotesize G. DAI, Z. SU, V. ULYANOV AND H. WANG }

\fancyhead [LE]{\thepage}
\fancyhead [RO]{\thepage}
\renewcommand{\headrulewidth}{0mm}

\author{Guozheng Dai}
\address{Zhejiang University, Hangzhou, 310058,  China.}
\email{11935022@zju.edu.cn}

\author{Zhonggen Su}
\address{Zhejiang University, Hangzhou, 310058,  China.}
\email{suzhonggen@zju.edu.cn}

\author{Vladimir Ulyanov}
\address{Moscow State University, Leninskie Gory, 1, Moscow, Russia.}
\email{vladim53@yandex.ru}

\author{Hanchao Wang }
\address{Shandong University,  Jinan,  250100, China.}
\email{wanghanchao@sdu.edu.cn}

\date{}


\begin{abstract}	
	In this paper, we obtain a  $p$-th moment bound for the suprema of a log-concave-tailed nonhomogeneous chaos process, which is optimal in some special cases. A crucial ingredient of the proof is a novel decoupling inequality, which may be of independent interest. With this $p$-th moment  bound, we show two  uniform Hanson-Wright type deviation inequalities for $\alpha$-subexponential entries ($1\le \alpha\le 2$), which recover some known results. As applications, we prove the restricted isometry property of partial random circulant matrices and time-frequency structured random matrices induced by standard $\alpha$-subexponential vectors ($1\le \alpha\le 2$), which extends the previously known results for the subgaussian case. 
\end{abstract}

\maketitle

\section{Introduction and Main Results}

 Let $\xi=(\xi_{1},\cdots, \xi_{n})$ be a random vector with  independent standard entries. Assume $f(x_{1}, \cdots, x_{n})$ is an  order two polynomial. Then, we call $f(\xi_{1}, \cdots, \xi_{n})$ a (order two) chaos.
 A chaos $f(\xi_{1}, \cdots, \xi_{n})$ is homogeneous if $f(\cdot)$ is an order two homogeneous  polynomial, i.e., a random variable with the following form
\begin{align}
	S_A(\xi)=\sum_{i,j=1}^{n}a_{ij}\xi_{i}\xi_{j},\nonumber
\end{align}
where $A=(a_{ij})_{n\times n}$ is a fixed matrix.
 Furthermore, if all diagnal entries of $A$ are zero, $S_{A}(\xi)$ is tetrahedral. A chaos is  nonhomogeneous if it is not homogeneous.   A homogeneous (resp. nonhomogeneous) chaos process is a family of homogeneous (resp. nonhomogeneous) chaoses, $\{S_{A}(\xi), A\in\mathcal{A} \}$,  where $\mathcal{A}$ is a  family of matrices.

In this paper, we focus on investigating the suprema of some nonhomogeneous chaos processes. Specifically,  we would like to find the optimal bounds for the quantity $Z_{\mathcal{A}}(\xi):=\sup_{A\in \mathcal{A}}\big( S_A(\xi)-\textsf{E}S_{A}(\xi)\big)$. Indeed, the chaos $S_{A}(\xi)-\textsf{E}S_{A}(\xi)$ is generated by the following nohomogeneous polynomial,
\begin{align}\label{Eq_polynomial}
	f(x_{1}, \cdots, x_{n})=\sum_{i\neq j}a_{ij}f_{1}(x_{i})f_{1}(x_{j})+\sum_{i}a_{ii}f_{2}(x_{i}),
\end{align}
where $f_{1}(x)=x, f_{2}(x)=x^{2}-1$ are Hermite polynomials. The upper bound for tails of $S_{A}(\xi)-\textsf{E}S_{A}(\xi)$ is also called the Hanson-Wright type inequality, which has numerous applications in high-dimensional probability, statistics, and random matrix theory \cite{highdimension}. We also refer interested readers to \cite{Arcones_Gine} and the references therein for more information about such chaos.

If   $\{\xi_{i}, i\ge 1\}$  is a sequence of independent Gaussian variables, Arcones and Gin\'{e} \cite{Arcones_Gine} develop a decoupling inequality (see \eqref{Gaussian_decoupling} below) and then show an optimal bound for $Z_{\mathcal{A}}(\xi)$. But their bound contains a term that is difficult to estimate. Adamczak et al \cite{Adamczak_AIHP_2} and Krahmer et al \cite{Rauhut_CPAM} obtain bounds that can be handled more easily in many situations.

As for the non-Gaussian case, the bound for $Z_{\mathcal{A}}(\xi)$  is very difficult to investigate.  One of the difficulties comes from the lack of decoupling inequalities.  Here, we focus on the case when $\{\xi_{i}, i\ge 1\}$  is a sequence of independent log-concave-tailed variables (see the definition in Section 1.1).  In this paper, we develop a novel decoupling inequality for the nonhomogeneous chaos $S_{A}(\xi)-\textsf{E}S_{A}(\xi)$ and then prove a  $p$-th moment bounds for $Z_{\mathcal{A}}(\xi)$. With this moment bound in hand, we obtain two uniform Hanson-Wright type deviation inequalities via classical concentration inequalities and chaining arguments.
 One of the deviation inequalities helps us prove the restricted isometry property (R.I.P.) of partial random circulant matrices and time-frequency structured random matrices induced by  random vectors with standard $\alpha$-subexponential entries (see \eqref{Eq_definition_alpha_subexponential} below for the definition), $1\le \alpha\le 2$, which extends the previously known results for the  subgaussian case \cite{Rauhut_CPAM}. 
  
  In the remainder of this section, we first present the $p$-th moment  bounds for  the suprema of a log-concaved-tailed chaos process  and then two uniform Hanson-Wright type deviation inequalities for $\alpha$-subexponential entries, $1\le \alpha\le 2$. Finally, we give the results of the R.I.P. for some structured random matrices.

  \subsection{Suprema of Chaos Processes}
  At the beginning of this subsection, we first introduce some notations that will be used next.
  We say a random variable $\xi_{1}$ is log-concave-tailed if $-\log \textsf{P}\{\vert\xi_{1}\vert\ge t \}$ is convex. An interesting case is  when $-\log \textsf{P}\{\vert\xi_{1}\vert\ge t \}=t^{\alpha}$ ($ \alpha\ge 1$), the symmetric Weibull variables with the scale parameter $1$ and the shape parameter $\alpha$.
  A random variable $\eta_{1}$ is $\alpha$-subexponential (or $\alpha$-sub-Weibull) if its tail probability can be bounded by the symmetric Weibull variable's (with the scale parameter $1$ and the shape parameter $\alpha$). In particular,
  \begin{align}\label{Eq_definition_alpha_subexponential}
  	\textsf{P}\{\vert \eta_{1}\vert\ge Kt  \}\le ce^{-ct^{\alpha}},\quad t\ge 0,
  \end{align}
  where $K$ is a parameter and $c$ is a universal constant. We often refer to $\eta_{1}$ as a subgaussian variable when $\alpha =2$. The $\alpha$-subexponential norm of $\eta_{1}$ is defined as follows:
  \begin{align}\label{Eq_definition_alphanorm}
  	\Vert \eta_{1}\Vert_{\Psi_{\alpha}}:=\inf\{ t>0:\textsf{E}\exp(\frac{\vert \eta_{1}\vert^{\alpha}}{t^{\alpha}})\le 2\}.
  \end{align}

 The concentration properties of $S_{A}(\xi)$ (defined as before) is a classic topic in probability. A well-known result is due to Hanson and Wright, claiming that if $\xi_{i}$ are independent centered subgaussian variables satisfying $\max_{i}\Vert \xi_{i}\Vert_{\Psi_{2}}\le L$ and $A$ is a symmetric matrix, then for all $t\ge 0$ (this version was presented in \cite{Rudelson_ecp})
 \begin{align}\label{Eq_hansonwright_classic}
 	\textsf{P}\big\{\vert S_{A}(\xi)-\textsf{E}S_{A}(\xi)\vert \ge t  \big\}\le 2\exp\Big(-c\min\Big\{\frac{t^{2}}{L^{4}\Vert A\Vert_{F}^{2}}, \frac{t}{L^{2}\Vert A\Vert_{l_{2}\to l_{2}}}  \Big\}   \Big).
 \end{align}
Here, $\Vert \cdot\Vert_{F}$  and $\Vert \cdot\Vert_{l_{2}\to l_{2}}$ are the Frobenius norm and the spectral norm of a matrix, respectively.

An important extension of \eqref{Eq_hansonwright_classic} is to consider the concentration properties of $Z_{\mathcal{A}}(\xi)$ (defined as before), where $\mathcal{A}$ is a family of fixed matrices. When $\xi_{i}$ are independent Rademacher variables (take values $\pm 1$ with equal probabilities) and the diagnal entries of symmetric matrix $A$ are $0$, it was shown for $t\ge 0$ in Talagrand's celebrated paper \cite{Talagrand_invention}
\begin{align}\label{Eq_concentration_Talagrand}
	&\textsf{P}\big\{\vert Z_{\mathcal{A}}(\xi)-\textsf{E}Z_{\mathcal{A}}(\xi)\vert \ge t  \big\}\nonumber\\
	\le& 2\exp\Big(-c\min\Big\{\frac{t^{2}}{(\textsf{E}\sup_{A\in \mathcal{A}} \Vert A\xi\Vert_{2})^{2}}, \frac{t}{\sup_{A\in\mathcal{A}}\Vert A\Vert_{l_{2}\to l_{2}}}  \Big\}   \Big).
\end{align}

Recently, Klochkov and Zhivotovskiy \cite{Klochkov_EJP} obtained a uniform Hanson-Wright type inequality for more general random variables via the entropy method. In particular, assume $\xi_{i}$ are independent centered subgaussian variables and $A$ is symmetric. Let $L^{\prime}=\big\Vert\max_{i}\vert\xi_{i}\vert\big\Vert_{\Psi_{2}}$, then 
\begin{align}\label{Eq_Klochkov}
	&\textsf{P}\big\{ Z_{\mathcal{A}}(\xi)-\textsf{E}Z_{\mathcal{A}}(\xi)\ge t   \big\}\nonumber\\
	\le&\exp  \Big( -c\min\Big\{ \frac{t^{2}}{L^{\prime 2}(\textsf{E}\sup_{A\in \mathcal{A}}\Vert A\xi\Vert_{2})^{2}}, \frac{t}{L^{\prime 2}\sup_{A\in \mathcal{A}}\Vert A\Vert_{l_{2}\to l_{2}}}  \Big\}  \Big),
\end{align}
where $t\ge \max\{ L^{\prime}\textsf{E}\sup_{A\in\mathcal{A}}\Vert A\xi\Vert_{2}, L^{\prime 2}\sup_{A\in\mathcal{A}}\Vert A\Vert_{l_{2}\to l_{2}} \}$.

We usually call the results of the form \eqref{Eq_concentration_Talagrand} and \eqref{Eq_Klochkov} (one-sided) concentration inequalities. Another similar bounds are one-sided and have a multiplicative constant before $\textsf{E}Z_{\mathcal{A}}(\xi)$ (or replace $\textsf{E}Z_{\mathcal{A}}(\xi)$ with its suitable upper bound), which are called deviation inequalities. In the study of probability theory, concentration inequalities are obviously more important than deviation inequalities. But in some applications (such as the compressive sensing appearing below), we usually care more about the order of the bound.  Hence, deviation inequalities also play an important role in such applications.

Next, we shall introduce some deviation inequalities. To this end, we first introduce some notations. Denote $\big(\textsf{E}\vert \xi_{1}\vert^{p}\big)^{1/p}$ by $\Vert \xi_{1}\Vert_{L_{p}}$ for a random variable $\xi_{1}$. Set $\alpha^{*}=\alpha/(\alpha-1)$ as the conjuate exponent of $\alpha$. Define $M_{F}(\mathcal{A}):=\sup_{A\in\mathcal{A}}\Vert A\Vert_{F}$ and $M_{l_{2}\to l_{\alpha^{*}}}(\mathcal{A}):=\sup_{A\in\mathcal{A}}\Vert A\Vert_{l_{2}\to l_{\alpha^{*}}}$ (see \eqref{Eq_norms_matrix} below). We also denote Talagrand's $\gamma_{\alpha}$-functional with respect to $(\mathcal{A}, \Vert\cdot\Vert_{l_{2}\to l_{\alpha^{*}}})$ by $\gamma_{\alpha}(\mathcal{A}, \Vert\cdot\Vert_{l_{2}\to l_{\alpha^{*}}})$ (see (\ref{Talagrand's_functional}) below). Let
  \begin{align}
  	\Gamma(\alpha, \mathcal{A})=\gamma_{2}(\mathcal{A}, \Vert\cdot\Vert_{l_{2}\to l_{2}})+\gamma_{\alpha}(\mathcal{A}, \Vert\cdot\Vert_{l_{2}\to l_{\alpha^{*}}})\nonumber
  \end{align}
  and 
  \begin{align}
  	U_{1}(\alpha)=&\Gamma(\alpha, \mathcal{A}) \big(\Gamma(\alpha, \mathcal{A})+M_{F}(\mathcal{A}) \big),\nonumber\\
  	 U_{2}(\alpha)=&M_{l_{2}\to l_{2}}(\mathcal{A})\big(\Gamma(\alpha, \mathcal{A})+M_{F}(\mathcal{A})\big) ,\nonumber\\
  	U_{3}(\alpha)=&M_{l_{2}\to l_{\alpha^{*}}}(\mathcal{A})\big(\Gamma(\alpha, \mathcal{A})+M_{F}(\mathcal{A})\big) .\nonumber
  \end{align}

Assume $\xi_{i}$ are independent centered subgaussian variables and let $L=\max_{i}\Vert\xi\Vert_{\Psi_{2}}$. Krahmer et al \cite{Rauhut_CPAM} show for $t\ge 0$,
\begin{align}\label{Eq_Krahmer}
	&\textsf{P}\Big\{ \sup_{A\in\mathcal{A}}\big\vert\Vert A\xi\Vert_{2}^{2}-\textsf{E}\Vert A\xi\Vert_{2}^{2}  \big\vert\ge CL^{2} U_{1}(2)+t  \Big\}\nonumber\\
	\le& 2\exp\Big(-c\min\Big\{\frac{t^{2}}{L^{4}U_{2}^{2}(2)}, \frac{t}{L^{2}M^{2}_{l_{2}\to l_{2}}(\mathcal{A})}\Big\}  \Big),
\end{align}  
where $\mathcal{A}$ is a family of  $m\times n$ fixed matrices.

Also, in the subgaussian case, Adamczak et al \cite{Adamczak_AIHP_2} prove for $t\ge 0$
\begin{align}\label{Eq_latala}
	\textsf{P}\Big\{ Z_{\mathcal{A}}(\xi)>CL^{2}M(\mathcal{A})+t\Big\}
	\le 2\exp\Big(-c\min\Big\{\frac{t^{2}}{L^{4}T_{1}^{2}(\mathcal{A})}, \frac{t}{L^{2}M_{l_{2}\to l_{2}}(\mathcal{A})}  \Big\} \Big).
\end{align}
Here,
\begin{align}
	T_{1}(\mathcal{A})=\max\Big\{ \sup_{\Vert x\Vert_{2}\le 1}\textsf{E}\sup_{A\in\mathcal{A}}\vert \sum_{i\neq j}x_{i}a_{ij}g_{j}\vert, M_{F}(\mathcal{A})  \Big\}\nonumber
\end{align}
and 
\begin{align}
	M(\mathcal{A})=\textsf{E}\sup_{A\in \mathcal{A}}\big\vert \sum_{i,j}a_{ij}g_{i}g_{j}-\textsf{E}\sum_{i,j}a_{ij}g_{i}g_{j}\big\vert +\textsf{E}\sup_{A\in \mathcal{A}}\vert \sum_{i\neq j}a_{ij}g_{ij}\vert\nonumber,
\end{align}
where $\{g_{i}, g_{ij}: 1\le i, j\le n \}$ is a sequence of independent standard Gaussian variables and $\mathcal{A}$ is a family of fixed symmetric matrices.

All the results we mentioned above except \eqref{Eq_concentration_Talagrand} consider the nonhomogeneous chaos generated by the polynomial \eqref{Eq_polynomial}. There are a lot of exciting papers exploring the properties of homogeneous tetrahedral chaoses, such as \cite{Adamczak_AIHP,Meller_Studia_math,Meller_arxiv}. We refer  interested readers to these papers and the references therein for more information.

  % The majorizing measure theorem ensures one can optimally bound the suprema of gaussian processes with the generic chaining argument. But this will no longer be the case in general. For some nongaussian processes, the generic chaining argument sometimes leads to a terrible upper bound; see an example in \cite{Bogucki}. Fortunalely, Talagrand \cite{Talagrand_ajm_prob} and Latała \cite{Latala_gafa} generalize the majorizing measure theorem to other classes of processes, certain canonical processes; see Lemma \ref{Lemma_chaining} and \ref{Lemma_chaining2}. These generalizations play an essential role in proving the following results. 
  
  Our main result provides a $p$-th moment bound for the suprema of a nonhomogeneous log-concave-tailed chaos process generated by \eqref{Eq_polynomial}. This bound can lead to some uniform Hanson-Wright type deviation inequalities for $\alpha$-subexponential entries, $1\le \alpha\le 2$.
  The main tools for proving this result are a novel decoupling inequality (see Proposition \ref{Prop_decoupling}), the generalized majorizing measure theorem for log-concave-tailed canonical processes (see Lemma \ref{Lemma_chaining}), and a chaining argument appearing in \cite{Latala_EJP,Mendelson_JFA}.

  \begin{mytheo}\label{Theo_logconcave_deviation}
  	Let $\xi=(\xi_{1},\cdots, \xi_{n})$ be a random vector whose entries are independent centered variables with log-concave tails, namely the functions $U_{i}(x):=-\log \textsf{P}\{\vert\xi_{i}\vert\ge x\}$ are convex. Let $\mathcal{A}$ be a family of fixed $n\times n$ matrices. Assume that there is a finite positive constant $\beta_{0}$ such that for any $i\ge 1, x\ge 1$,
  	\begin{align}\label{Eq_regular_condition}
  		U_{i}(2x)\le \beta_{0}U_{i}(x), \quad U_{i}^{'}(0)\ge \frac{1}{\beta_{0}}.
  	\end{align}
  	Then, we have for $p\ge 1$
  	\begin{align}
  		\big\Vert \sup_{A\in \mathcal{A}}\big\vert S_{A}(\xi)-\textsf{E}S_{A}(\xi)\big\vert\big\Vert_{L_{p}}
  		\lesssim_{\beta_{0}}\big\Vert \textsf{E}_{\xi}\sup_{A\in \mathcal{A}}\big\vert \xi^\top  A\tilde{\xi}\big\vert \big\Vert_{L_{p}}+\sup_{A\in\mathcal{A}}\big\Vert  \xi^\top  A\tilde{\xi}\big\Vert_{L_{p}},\nonumber
  	\end{align}
  	where $\tilde{\xi}$ is an independent copy of $\xi$ and $\textsf{E}_{\xi}$ means that we only take expectations with respect to $\xi$.
  \end{mytheo}
  \begin{myrem}
  	 The regularity condition (\ref{Eq_regular_condition}) for $U_{i}(x)$ ensures that the functions do not grow too fast, which is needed in the generalized majorizing measure theorem (see Lemma \ref{Lemma_chaining}).
  \end{myrem}
 
  Theorem \ref{Theo_logconcave_deviation} yields the following uniform Hanson-Wright deviation ineauality with some classic concentration inequalities (see Theorem 4.19 in \cite{Ledoux_book}). 
 
  \begin{mycor}\label{Cor_logconcave}
  	Suppose that the entries of $\xi=(\xi_{1}, \cdots,\xi_{n})^\top$ are independent centered $\alpha$-subexponential variables ($1\le \alpha\le 2$) and $\eta=(\eta_{1},\cdots, \eta_{n})$ is a random vector whose entries are i.i.d. random variables with density $c(\alpha)\exp(-\vert x\vert^{\alpha})$ ($c(\alpha)$ is a normalization constant). Let $\mathcal{A}$ be a set containing fixed $n\times n$ symmetric  matrices. Set $L=\max_{i}\Vert \xi_{i}\Vert_{\Psi_{\alpha}}$ and $$T(\mathcal{A})=\max\Big\{ \textsf{E}\sup_{A\in\mathcal{A}}\Vert A\eta\Vert_{2}, M_{F}(\mathcal{A})  \Big\}.$$
  	Then, we have for $p\ge 1$
  	\begin{align}
  			\Big\Vert \sup_{A\in\mathcal{A}}\big\vert S_{A}(\xi)-\textsf{E}S_{A}(\xi)\big\vert \Big\Vert_{L_{p}}\lesssim_{\alpha, L}&\textsf{E}\sup_{A\in\mathcal{A}}\vert \eta^\top  A\tilde{\eta}\vert+\sqrt{p}T(\mathcal{A})\nonumber\\
  			&+p^{1/\alpha}\textsf{E}\sup_{A\in\mathcal{A}}\Vert A\eta\Vert_{\alpha^{*}}+p^{2/\alpha}M_{l_{2}\to l_{2}}(\mathcal{A})\nonumber.
  	\end{align}
  	Moreover, we have for $t\ge 0$
  	\begin{align}\label{Eq_1.4}
  		&\textsf{P}\Big\{  \sup_{A\in \mathcal{A}}\big\vert S_{A}(\xi)-\textsf{E}S_{A}(\xi) \big\vert>C(\alpha)L^{2}\Big(\textsf{E}\sup_{A\in\mathcal{A}}\vert \eta^\top  A\tilde{\eta}\vert+t\Big) \Big\}\nonumber\\
  		\le& C_{1}(\alpha)\exp\Big( -\min\Big\{\big(\frac{t}{T(\mathcal{A})}\big)^{2}, \big(\frac{t}{\textsf{E}\sup_{A\in\mathcal{A}}\Vert A\eta\Vert_{\alpha^{*}}}\big)^{\alpha},  \big(\frac{t}{M_{l_{2}\to l_{2}}(\mathcal{A})} \big)^{\alpha/2}\Big\}    \Big),
  	\end{align}
  	where  $\tilde{\eta}$ is an independent copy of $\eta$ and $M_{l_{2}\to l_{2}}(\mathcal{A})=\sup_{A\in\mathcal{A}}\Vert A\Vert_{l_{2}\to l_{2}}$ as defined before.
  \end{mycor}

We can further bound $\textsf{E}\sup_{A\in\mathcal{A}}\vert \eta^\top  A\tilde{\eta}\vert$ and $\textsf{E}\sup_{A\in\mathcal{A}}\Vert A\eta\Vert_{\alpha^{*}}$ appearing in Corollary \ref{Cor_logconcave} with $\gamma_{\alpha}$-functionals via chaining arguments, which is the next result. This result is a key ingredient to show the R.I.P. for the structured random matrices. Except for the application in proving the R.I.P. of some random matrices,  such type deviation inequalities are also used to study the dimensionality reduction of sparse Johnson–Lindenstrauss transform (see Bourgain et al. \cite{Bougain_gafa}). 

Recall the notations $U_{i}(\alpha), i=1, 2, 3$. The next result reads as follows:
\begin{mycor}\label{Theo_deviation}
	Let $\mathcal{A}$ be a family of fixed $m\times n$ matrices and $\xi$ be a random vector as in Corollary \ref{Cor_logconcave}. We have for $p\ge 1$
	\begin{align}
		\Big\Vert  \sup_{A\in \mathcal{A}}\big\vert \Vert A\xi\Vert_{2}^{2}-\textsf{E}\Vert A\xi\Vert_{2}^{2}  \big\vert\Big\Vert_{L_{p}}\lesssim_{\alpha, L}&U_{1}(\alpha)+\sqrt{p}U_{2}(\alpha)\nonumber\\
		&+p^{1/\alpha}U_{3}(\alpha)+p^{2/\alpha}M^{2}_{l_{2}\to l_{2}}(\mathcal{A}).\nonumber
	\end{align}
	 Moreover, we have for $t\ge 0$
	\begin{align}\label{Eq_remark1.6}
		&\textsf{P}\Big\{ \sup_{A\in \mathcal{A}}\big\vert \Vert A\xi\Vert_{2}^{2}-\textsf{E}\Vert A\xi\Vert_{2}^{2}  \big\vert>C(\alpha)L^{2}(U_{1}(\alpha)+t)\Big\}\nonumber\\
		\le& C_{1}(\alpha)\exp\Big(-\min\Big\{(\frac{t}{U_{2}(\alpha)})^{2},( \frac{t}{U_{3}(\alpha)})^{\alpha}, (\frac{t}{M_{l_{2}\to l_{2}}^{2}(\mathcal{A})})^{\alpha/2}\Big\}\Big).
	\end{align}
\end{mycor}

\textbf{Discussions for our results:}

(i) If all matrices in $\mathcal{A}$ are symmetric and their diagonal entries are 0, the upper bound in Theorem \ref{Theo_logconcave_deviation} is optimal (up to a universal constant). Indeed, we have by Theorem 3.1.1 in \cite{Gine_decoupling_book}
\begin{align}
	\big\Vert \sup_{A\in \mathcal{A}}\big\vert S_{A}(\xi)-\textsf{E}S_{A}(\xi)\big\vert\big\Vert_{L_{p}}\asymp& \big\Vert \sup_{A\in \mathcal{A}}\big\vert \xi^\top A\tilde{\xi}\big\vert\big\Vert_{L_{p}}\nonumber\\
	\ge& \max\Big\{\big\Vert \textsf{E}_{\xi}\sup_{A\in \mathcal{A}}\big\vert \xi^\top  A\tilde{\xi}\big\vert \big\Vert_{L_{p}}, \sup_{A\in\mathcal{A}}\big\Vert  \xi^\top  A\tilde{\xi}\big\Vert_{L_{p}}\Big\}.\nonumber
\end{align}

(ii) Assume that $\mathcal{A}$ contains only one symmetric matrix and $\xi_{1}, \cdots, \xi_{n}\stackrel{i.i.d.}{\sim}\mathcal{W}_{s}(\alpha)$ (i.e. $\xi_{i}$ are symmetric and $\textsf{P}\{\vert\xi_{i}\vert>t  \}=e^{-t^{\alpha}}$), $1\le \alpha\le 2$. Combined with Lemmas \ref{Lem_Moments} and \ref{Lem_alpha_2}, Theorem \ref{Theo_logconcave_deviation} yields for $t\ge 0$
\begin{align}\label{Eq_remark_ii}
	\textsf{P}\big\{ \vert \xi^\top A\xi-\textsf{E}\xi^\top A\xi\vert> C(\alpha)L^{2}t\big\}\le C_{1}(\alpha)\exp\big(-\phi_{2}(A, \alpha, t)\big),
\end{align}
where 
\begin{align}
	\phi_{2}(A, \alpha, t)
	=\min\Big\{&(\frac{t}{\Vert A\Vert_{F}})^{2}, \frac{t}{\Vert A\Vert_{l_{2}\to l_{2}}},\nonumber\\
	& (\frac{t}{\Vert A\Vert_{l_{\alpha^{*}}(l_{2})}})^{\alpha} , (\frac{t}{\Vert A\Vert_{l_{2}\to l_{\alpha^{*}}}})^{\frac{2\alpha}{\alpha+2}}, (\frac{t}{\Vert A\Vert_{l_{\alpha}\to l_{\alpha^{*}}}})^{\frac{\alpha}{2}}\Big\}.\nonumber
\end{align}
\eqref{Eq_remark_ii} is also valid in the case where $\xi_{i}$ are independent centered $\alpha$-subexponential variables ($1\le \alpha\le 2$) due to Proposition \ref{Prop_decoupling}, which recovers the classic Hanson-Wright inequality (\ref{Eq_hansonwright_classic}) when $\alpha=2$. One can refer to \eqref{Eq_norms_matrix} for the matrix norms appearing above.

  	(iii) Compared with (\ref{Eq_Klochkov}), Corollary \ref{Cor_logconcave} (when $\alpha=2$) can behaves better in some cases. For example, let $\xi_{i}, 1\le i\le n$ be i.i.d. standard Gaussian variables and  $\mathcal{A}$ contains only one identity matrix. Note that $\max_{i}\Vert\xi_{i}\Vert_{\Psi_{2}}=\sqrt{8/3}$ and $\big\Vert\max_{i}\vert\xi_{i}\vert\big\Vert_{\Psi_{2}}\asymp\sqrt{\log n}$. A direct calculation yields $\textsf{E}\Vert \xi\Vert_{2}\asymp \sqrt{n}$.

  	Hence, \eqref{Eq_Klochkov} yields that
  	\begin{align}
  		\Big\Vert  \Vert \xi\Vert_{2}^{2} \Big\Vert_{L_{p}}\lesssim n+ \sqrt{pn\log n}+p\log n,\nonumber
  	\end{align}
  which has an extra $\log n$ term compared with the bound deduced by Corollary \ref{Cor_logconcave}. 	
  	
  	(iv) Let $\{g_{i}, g_{ij}: 1\le i, j\le n\}$ be i.i.d. standard normal variables and $A=(a_{ij})_{n\times n}$ be a fixed symmetric matrix. Let $g=(g_{1}, \cdots, g_{n})$ and $\tilde{g}$ be an independent copy of $g$.
   Compared with the result in Corollary \ref{Cor_logconcave} (when $\alpha=2$), \eqref{Eq_latala} has a better tail probability due to that
  \begin{align}
  	\sup_{\Vert x\Vert_{2}\le 1}\textsf{E}\sup_{A\in\mathcal{A}}\vert \sum_{i\neq j}x_{i}a_{ij}g_{j}\vert\le \textsf{E}\sup_{A\in \mathcal{A}}\Vert  Ag\Vert_{2}.\nonumber
  \end{align}
  However, the advantage of Corollary \ref{Cor_logconcave} is that $$\textsf{E}\sup_{A\in \mathcal{A}}\big\vert S_{A}(g)-\textsf{E}S_{A}(g)\big\vert +\textsf{E}\sup_{A\in \mathcal{A}}\vert \sum_{i\neq j}a_{ij}g_{ij}\vert\gtrsim \textsf{E}\sup_{A\in\mathcal{A}}\vert g^\top A\tilde{g}\vert,$$ which is due to that (see (2.9) in \cite{Arcones_Gine})
  \begin{align}
  	\textsf{E}\sup_{A\in \mathcal{A}}\big\vert S_{A}(g)-\textsf{E}S_{A}(g)\big\vert\asymp\textsf{E}\sup_{A\in\mathcal{A}}\vert g^\top  A\tilde{g}\vert.\nonumber
  \end{align}

(v) In the case $\alpha=2$, Corollary \ref{Theo_deviation} recovers the result \eqref{Eq_Krahmer}.

\subsection{ R.I.P. of Structured Random Matrices}
\subsubsection{Compressive Sensing}

Compressive sensing is a technique to recover sparse vectors in high dimensions from incomplete information with efficient algorithms. It has been applied in many aspects, such as signal and image processing \cite{Raught_book,Raught_zongshu}.

Recall that, a vector $x\in \mathbb{R}^{n}$ is $s$-sparse if satisfying $\Vert x\Vert_{0}:= \vert \{l: x_{l}\neq0  \}\vert\le s$. The recovery task is to obtain $x$ from the observed information
\begin{align}
	y=\Phi x,\nonumber
\end{align}
where $\Phi\in \mathbb{R}^{m\times n}$ is given and called  the measurement matrix. Note that this system is determined when $m=n$. Hence, we are interested in the case $m\ll n$. One of the main recovery approaches is based on $l_{1}$-minimization, 
\begin{align}
	\min_{z} \Vert z\Vert_{1} \quad \text{s.t.} \quad\Phi z=y,\nonumber
\end{align}
where $\Vert z\Vert_{p}$ is the $l_{p}$-norm. A natural question arises: whether the solution to the above convex optimization problem implies an exact recovery. 

We say that, an $m\times n$ matrix $A$ possesses the restricted isometry property (R.I.P.) with parameters $\delta$ and $s$ if satisfying for all $s$-sparse $x$
\begin{align}\label{Eqisometry}
	(1-\delta)\Vert x\Vert_{2}^{2}\le \Vert A x\Vert_{2}^{2}\le (1+\delta)\Vert x\Vert_{2}^{2}.
\end{align}
And we define the restricted isometry constant $\delta_{s}$ as the smallest number satisfying the inequality (\ref{Eqisometry}). Cand\`{e}s and Tao \cite{Tao_recovery} show that the R.I.P. of the measurement matrix $\Phi$ implies an exact recovery.

It is important to note that most of the known optimal measurement matrices are random matrices. For example, Gaussian or Bernoulli random matrices or, more generally, subgaussian random matrices. The entries of a Bernoulli random matrix are i.i.d. variables taking values $1/\sqrt{m}$ and $-1/\sqrt{m}$ with equal probabilities, and a Gaussian matrix has entries as i.i.d. random normal variables  with mean $0$ and variance $1/m$. With high probability, the above random matrices satisfy the R.I.P.  \cite{Gaussian_Proof},  allowing an exact sparse recovery with $l_{1}$-minimization.

Although Gaussian and Bernoulli matrices allow an exact sparse recovery according to $l_{1}$-minimization, they have limited use in applications. On the one hand, physical or other constraints of the application define the measurement matrix without leaving us the freedom to design anything. Hence, it is often not justifiable that the matrix entries are i.i.d. Gaussian or Bernoulli variables. On the other hand, no fast matrix multiplication is available for such unstructured matrices, making the recovery algorithms run very slowly. Hence, Gaussian or Bernoulli matrices are useless for large scale problems.

Due to these observations, in this paper we consider measurement matrices with the following features. Firstly, we study the measurement matrix with $\alpha$-subexponential ($1\le \alpha\le 2$) entries. 
We note that a random matrix with independent standard $\alpha$-subexponential ($1\le \alpha\le 2$) entries satisfies R.I.P. with high probability (see a detailed proof in Section 5.1)

Secondly, we study random matrices with more general structure. In particular, we consider two types of structured random matrices, namely partial random circulant matrices and time-frequency structured random matrices, which have been studied in several works, including \cite{Rauhut_CPAM,Rauhut_ACHA}.

\subsubsection{Partial Random Circulant Matrices}

Recall the definition of cyclic subtraction: $j\ominus k=(j-k)\mod n$. Denote the circular convolution of two vectors $x, y\in \mathbb{R}^{n}$ by $z*x=((z*x)_{1}, \cdots, (z*x)_{n})^\top$, where
\begin{align}
	(z*x)_{j}:=\sum_{k=1}^{n}z_{j\ominus k}x_{k}.\nonumber
\end{align}
We define the circulant matrix $H=H_{z}\in \mathbb{R}^{n\times n}$ generated by $z$ with entries $H_{jk}=z_{j\ominus k}$. Equivalently, $Hx=z*x$ for every $x\in \mathbb{R}^{n}$. 

Let $\Omega\subset \{1,\cdots, n\}$ be a fixed set with $\vert \Omega\vert=m$. Denote by $R_{\Omega}: \mathbb{R}^{n}\to \mathbb{R}^{m}$  the operator restricting a vector $x\in \mathbb{R}^{n}$ to its entries in $\Omega$. We define the partial circulant matrix generated by $z$ as follows
\begin{align}
	\Phi=\frac{1}{\sqrt{m}}R_{\Omega}H_{z}.\nonumber
\end{align}
If the generating vector $z$ is random, we call $\Phi$ a partial random circulant matrix. 

Let $\varepsilon=(\varepsilon_{1},\cdots, \varepsilon_{n})^\top$ be a Rademacher vector: that is a random vector with independent Rademacher entries. Rauhut, Romberg, and Tropp \cite{Rauhut_ACHA} proved that the $m\times n$ partial random circulant matrix $\Phi$ generated by $\varepsilon$ satisfies R.I.P., that is for fixed $s\le n$ and $\delta\in (0, 1)$, the restricted isometry constants satisfy $\delta_{s}\le \delta$ with high probability if 
\begin{align}
	m\ge C\max\{\delta^{-1}s^{3/2}\log^{3/2}n, \delta^{-2}s\log^{2}n\log^{2}s  \}.\nonumber
\end{align}

Krahmer, Mendelson, and Rauhut \cite{Rauhut_CPAM} improved the above result. They proved that 
\begin{align}
	m\ge C\delta^{-2}s\log^{2}n\log^{2}s\nonumber
\end{align}
is  enough to ensure $\delta_{s}\le \delta$ with high probability. In fact, they consider a more general case when $\Phi$ is generated by a random vector $\xi$ whose entries are independent sub-Gaussian variables with mean $0$ and variance $1$.

The following result establishes the R.I.P. for partial random circulant matrices generated by a standard $\alpha$-subexponential random vector $\eta=(\eta_{1},\cdots,\eta_{n})^\top$ ($1\le \alpha\le 2$), i.e., $\{ \eta_{i}, i\le n\}$ are  independent $\alpha$-subexponential variables with mean $0$ and variance $1$:

\begin{mytheo}\label{Theo_RIP}
	Let $\Phi$ be an $m\times n$ partial random circulant matrix generated by a standard $\alpha$-subexponential random vector $\eta$, $1\le \alpha\le 2$. Set $L=\max_{i\le n}\Vert \eta_{i}\Vert_{\Psi_{\alpha}}$. Then, under the condition $m\ge c_{1}(\alpha, L)\delta^{-2}f_{1}(s, n)$, the restricted isometry constant $\delta_{s}$ satisfies
	\begin{align}
		\textsf{P}\{ \delta_{s}\le \delta \}\ge 1-\exp(-c_{0}(\alpha, L)f_{2}(s, n)).       \nonumber
	\end{align}
	Here, 
	\begin{align}
		f_{1}(s, n)=\max\big\{(s^{2/\alpha}\log^{4/\alpha} n),     (s\log^{2}s\log^{2}n)\big\}\nonumber
	\end{align}
	and 
	\begin{align}
		f_{2}(s, n)=\max\big\{ (s^{(2-\alpha)/2}\log^{2} n), (\log^{\alpha}s\log^{\alpha}n)  \big\} .\nonumber 
	\end{align}
	
\end{mytheo}
For the case $\alpha=2$, Theorem \ref{Theo_RIP} yields for $m\ge c_{0}(L)\delta^{-2}s\log^{2}s\log^{2}n$
\begin{align}
	\textsf{P}\{ \delta_{s}\le \delta \}\ge 1-n^{-c_{1}( L)\log n\log ^{2}s},\nonumber
\end{align}
where $c_{0}(L), c_{1}(L)$ are constants depending only on $L$. Thus our result recovers the previous result (see Theorem 1.1 in \cite{Rauhut_CPAM}) of Krahmer, Mendelson, and Rauhut.

\subsubsection{Time-Frequency Structured Random Matrices}
Let $h\in \mathbb{C}^{m}$ be a fixed vector. Define operators $T, M$
\begin{align}
	(Th)_{j}=h_{j\ominus 1},\quad (Mh)_{j}=e^{2\pi \boldsymbol{i}  j/m}h_{j}=\omega^{j}h_{j},\nonumber
\end{align}
where $j\ominus k=(j-k)\,\,\text{mod}\,\, m$, $\boldsymbol{i}$ is the imaginary unit and $\omega=e^{2\pi \boldsymbol{i}  /m}$. The time-frequency shifts are given by
\begin{align}
	\boldsymbol{\pi}(\lambda)=M^{l}T^{k},\quad \lambda=(k, l)\in \{0, \cdots, m-1 \}^{2}.\nonumber
\end{align}
For $h\neq 0\in \mathbb{C}^{m}$, the time-frequency structured random matrix generated by $h$ is the $m\times m^{2}$ matrix $\Psi_{h}$ whose columns are the vectors $\boldsymbol{\pi}(\lambda)(h)$ for $\lambda\in \{0, \cdots, m-1\}^{2}$. Here, the signal length $n$ is coupled to the embedding dimension $m$ via $n=m^{2}$. 

Krahmer, Mendelson, and Rauhut \cite{Rauhut_CPAM} established the restricted isometry property for the time-frequency structured random matrix generated by a normalized standard sub-Gaussian vector. In particular, let $\xi$ be a standard sub-Gaussian vector and consider the time-frequency structured random matrix $\Psi_{h}\in \mathbb{C}^{m\times m^{2}}$ generated by $h=(1/\sqrt{m})\xi$. If
$
m\ge c\delta^{-2} s\log^{2}s\log^{2}m,\nonumber 
$
then the restricted isometry constant of $\Psi_{h}$ satisfies 
\begin{align}
	\textsf{P}\{ \delta_{s}\le \delta  \}\ge  1-m^{-\log m\log^{2}s}.\nonumber
\end{align}

We refer interested readers to \cite{Rauhut_CPAM} and the references therein for more discussions of time-frequency structured random matrices.

Our following result extends the above result, showing the  restricted isometry property for the time-frequency structured random matrix generated by a normalized standard $\alpha$-subexponential vector, $1\le \alpha\le 2$.

\begin{mytheo}\label{Theo_time_frequency}
	Denote $\eta=(\eta_{1},\cdots,\eta_{m})$ a standard $\alpha$-subexponential vector. Let $\Psi_{h}$ be an $m\times m^{2}$ time-frequency structured random matrix generated by  $h=\eta/\sqrt{m}$, $1\le \alpha\le 2$. Set $L=\max_{i\le n}\Vert \eta_{i}\Vert_{\Psi_{\alpha}}$. Then, under the condition $m\ge c_{1}(\alpha, L)\delta^{-2}f_{1}(s, m)$ the restricted isometry constant $\delta_{s}$ satisfies
	\begin{align}
		\textsf{P}\{ \delta_{s}\le \delta \}\ge 1-\exp(-c_{0}(\alpha, L)  f_{2}(s, m)),    \nonumber
	\end{align}
where the functions $f_{1}, f_{2}$ are as defined in Theorem \ref{Theo_RIP}.
\end{mytheo}

\subsection{Organization of the paper.} Section 2 will introduce some notations and auxiliary lemmas, mainly including the moment bounds for decoupled chaos variables, the contraction principle, and the majorizing measure theorems. In Section 3, we shall prove our main results. We first prove Theorem \ref{Theo_logconcave_deviation}, the $p$-th moment bounds for the suprema of chaos processes, in Section 3.1. Then, we prove  Corollaries \ref{Cor_logconcave} and \ref{Theo_deviation} in Section 3.2.  At last, we conclude the proof of Theorems \ref{Theo_RIP} and \ref{Theo_time_frequency} in Sections 3.3 and 3.4. In Section 4, we show an improved result of Corollary \ref{Theo_deviation} and then provide a possible way to improve Theorems \ref{Theo_RIP} and \ref{Theo_time_frequency}.
 In the Appendix, we first provide a proof of the R.I.P. of $\alpha$-subexponential random matrices. Then,  some integration operations have been added when estimating the bounds of $\gamma_{\alpha}$-functionals.

\section{Preliminaries}

\subsection{Notations}

For a fixed vector $x=(x_{1},\cdots, x_{n})^\top\in \mathbb{R}^{n}$, we denote by $\Vert x\Vert_{p}=(\sum\vert x_{i}\vert^{p})^{1/p}$ the $l_{p}$ norm. We use $\Vert \xi\Vert_{L_{p}}=(\textsf{E}\vert \xi\vert^{p})^{1/p}$ for the $L_{p}$ norm of a random variable $\xi$.
As for an $m\times n$ matrix $A=(a_{ij})$, recall the following notations of norms
\begin{align}\label{Eq_norms_matrix}
	\Vert A\Vert_{F}=&\sqrt{\sum_{i,j}\vert a_{ij}\vert^{2}}, \quad \Vert A\Vert_{\infty}=\max_{ij}\vert a_{ij}\vert, \quad \Vert A\Vert_{l_{p}(l_{2})}=(\sum_{i\le m}(\sum_{j\le n}\vert a_{ij}\vert^{2})^{p/2})^{1/p}\nonumber\\
	&\Vert A\Vert_{l_{p_{1}}\to l_{p_{2}}}=\sup \{\vert\sum a_{ij}x_{j}y_{i}\vert: \Vert x\Vert_{p_{1}}\le 1, \Vert y\Vert_{p_{2}^{*}} \le 1  \},
\end{align}
where $p_{2}^{*}=p_{2}/(p_{2}-1)$. We remark that $\Vert A\Vert_{l_{2}\to l_{2}}$ is the important spectral norm of $A$ and $\Vert A\Vert_{l_{1}\to l_{\infty}}=\Vert A\Vert_{\infty}$.

Unless otherwise stated, we denote by $C, C_{1}, c, c_{1},\cdots$ universal constants which are independent of any parameters, and by $C(\delta), c(\delta),\cdots $ constants that depend only on the parameter $\delta$. 
For convenience, we say $f\lesssim g$ if $f\le Cg$ for some universal constant $C$ and say $f\lesssim_{\delta} g$ if $f\le C(\delta)g$ for some constant $C(\delta)$. We also say $f\asymp$ if $f\lesssim g$ and $g\lesssim f$, so does $f\asymp_{\delta} g$. We denote $\xi\sim \mathcal{W}_{s}(\alpha)$ if $\xi$ is  a symmetric Weibull variable with the scale parameter $1$ and the shape parameter $1 \le \alpha \le 2$. In particular, $-\log \textsf{P}\{\vert\xi\vert>x  \}=x^{\alpha}, x\ge 0$.

\subsection{Tails and Moments}

This subsection will introduce some basic properties of the tails and moments for some random variables.

\begin{mylem}\label{Lem_Moments}
	Assume that a random variable $\xi$ satisfies for $p\ge p_{0}$
	\begin{align}
		\Vert \xi\Vert_{L_p}\le \sum_{k=1}^{m}C_{k}p^{\beta_{k}}+C_{m+1},\nonumber
	\end{align}
	where $C_{1},\cdots, C_{m+1}> 0$ and $\beta_{1},\cdots, \beta_{m}>0$. Then we have for any $t>0$,
	\begin{align}\label{Eq_Lemma_2.2_1}
		\textsf{P}\big\{ \vert \xi\vert>e(mt+C_{m+1}) \big\}\le e^{p_{0}}\exp\Big(-\min\big\{\big(\frac{t}{C_{1}}\big)^{1/\beta_{1}},\cdots, \big(\frac{t}{C_{m}}\big)^{1/\beta_{m}}\big\}\Big)
	\end{align}
	and
	\begin{align}\label{Eq_Lemma_2.2_2}
		\textsf{P}\Big\{\vert \xi\vert>e \big(\sum_{k=1}^{m}C_{k}t^{\beta_{k}}+C_{m+1}\big)  \Big\}\le e^{p_{0}}e^{-t}.
	\end{align}
\end{mylem}

\begin{proof}
	We first prove \eqref{Eq_Lemma_2.2_1}.	Consider the following function,
	\begin{align}
		f(t):=\min\big\{\big(\frac{t}{C_{1}}\big)^{1/\beta_{1}},\cdots, \big(\frac{t}{C_{m}}\big)^{1/\beta_{m}}\big\}.\nonumber
	\end{align}
	If we assume $f(t)\ge p_{0}$, we can estimate 
	\begin{align}
		\Vert \xi\Vert_{L_{f(t)}}\le \sum_{k=1}^{m}C_{k}f(t)^{\beta_{k}}+C_{m+1}\le mt+C_{m+1}.\nonumber
	\end{align}
	Hence, we have by Markov's inequality
	\begin{align}
		\textsf{P}\{\vert \xi\vert>e(mt+C_{m+1})  \}\le \textsf{P}\{\vert \xi\vert >e \Vert \xi\Vert_{f(t)}  \}\le e^{-f(t)}.\nonumber
	\end{align}
	As for $f(t)<p_{0}$, we have the following trival bound
	\begin{align}
		\textsf{P}\{\vert \xi\vert>e(mt+C_{m+1})  \}\le 1\le e^{p_{0}}e^{-f(t)}.\nonumber
	\end{align}
	Hence, we have for $t\ge 0$
	\begin{align}
		\textsf{P}\{\vert \xi\vert>e(mt+C_{m+1})  \}\le  e^{p_{0}}e^{-f(t)}.\nonumber
	\end{align}
	
	Then, we turn to the proof of \eqref{Eq_Lemma_2.2_2}. By Markov's inequality, we have for any $p\ge p_{0}$
	\begin{align}
		\textsf{P}\Big\{\vert \xi\vert>e \big(\sum_{k=1}^{m}C_{k}t^{\beta_{k}}+C_{m+1}\big) \Big \}&\le \frac{\textsf{E}\vert \xi\vert^{p}}{e^{p} (\sum_{k=1}^{m}C_{m}t^{\beta_{k}}+C_{m+1})^{p}}\nonumber\\
		&\le \left(\frac{\sum_{k=1}^{m}C_{k}p^{\beta_{k}}+C_{m+1}}{e (\sum_{k=1}^{m}C_{k}t^{\beta_{k}}+C_{m+1})}\right)^{p}.\nonumber
	\end{align}
	Letting $p=t$, we have for $t\ge p_{0}$
	\begin{align}
		\textsf{P}\Big\{\vert \xi\vert>e \big(\sum_{k=1}^{m}C_{m}t^{\beta_{k}}+C_{m+1}\big) \Big \}\le e^{-t}.\nonumber
	\end{align}
	Moreover, if we further bound $e^{-t}$ by $e^{-(t-p_{0})}$, then the inequality is obviously valid for $t>0$, as desired.
\end{proof}

\begin{mylem}[Lemma 4.6 in \cite{Latala_Inventions}]\label{Lem_low bound}
	Let $\xi$ be a random variable such that
	\begin{align}
		K_{1}p^{\beta}\le \Vert\xi\Vert_{L_{p}}\le K_{2}p^{\beta}, \quad\text{for all}\,\, p\ge 2.\nonumber
	\end{align}
	Then there exist constants $K_{3}, K_{4}$ depending only on $K_{1}, K_{2}, \beta$ such that
	\begin{align}
		K_{3}e^{-t^{1/\beta}/K_{3}}\le \textsf{P}\{\vert \xi\vert\ge t\}\le K_{4}e^{-t^{1/\beta}/K_{4}},\quad\text{for all}\,\, t\ge 0.\nonumber
	\end{align}
\end{mylem}

\subsection{Decoupled Chaos Variables}

In this subsection, we collect some optimal moment estimates for decoupled chaos variables generated by a sequence of i.i.d. variables $\{\xi_{i}, i\le n\}$ with distribution $\mathcal{W}_{s}(\alpha)$, $1\le \alpha\le 2$. 
 For simplicity, we only present the results on decoupled chaos of order $1$ and $2$ here. 

\begin{mylem}[Theorem 6.1 in \cite{Adamczak_PTRF}]\label{Lem_alpha_2}
	For $p\ge 2$, we have
	
	(i)\begin{align}
		\big\Vert \sum_{i\le n}a_{i}\xi_{i}\big\Vert_{L_{p}}\asymp_{\alpha}p^{1/2}\Vert a\Vert_{2}+p^{1/\alpha}\Vert a\Vert_{\alpha^{*}},\nonumber
	\end{align}
 where $a=(a_{1},\cdots, a_{n})^\top$ is a fixed vector and $\alpha^{*}=\alpha/(\alpha-1).$
 
 (ii)\begin{align}\label{Eq_Lemma_optimalbound_2}
 	\big\Vert \sum_{i, j}a_{ij}\xi_{i}\tilde{\xi}_{j}\big\Vert_{L_{p}}\asymp_{\alpha}&p^{1/2}\Vert A\Vert_{F}+p\Vert A\Vert_{l_{2}\to l_{2}}+p^{1/\alpha}\Vert A\Vert_{l_{\alpha^{*}}(l_{2})},\nonumber\\
 	&+p^{(\alpha+2)/2\alpha}\Vert A\Vert_{l_{2}\to l_{\alpha^{*}}}+p^{2/\alpha}\Vert A\Vert_{l_{\alpha}\to l_{\alpha^{*}}},
 \end{align}
where $A=(a_{ij})$ is a fixed symmetric matrix and $\{\tilde{\xi}_{i}, i\le n\}$ are independent copies of $\{\eta_{i}, i\le n\}$.
\end{mylem}

\begin{myrem}\label{Rem_explanations}
Although the bound \eqref{Eq_Lemma_optimalbound_2} is optimal, but it is quite complex. Note that
\begin{align}
p^{1/\alpha}\Vert A\Vert_{l_{\alpha^{*}}(l_{2})}\le p^{1/2}\Vert A\Vert_{F},\quad	\Vert A\Vert_{l_{\alpha}\to l_{\alpha^{*}}}\le 	\Vert A\Vert_{l_{2}\to l_{\alpha^{*}}}\le 	\Vert A\Vert_{l_{2}\to l_{2}}.\nonumber
\end{align}
Here, one can refer to \cite{Sambale_notes} for a similar proof of the first inequality.
Hence, we have  for $p\ge 2$
\begin{align}
	\big\Vert \sum_{i, j}a_{ij}\xi_{i}\tilde{\xi}_{j}\big\Vert_{L_{p}}\lesssim_{\alpha}p^{1/2}\Vert A\Vert_{F}+p^{2/\alpha}\Vert A\Vert_{l_{2}\to l_{2}}.\nonumber
\end{align}

\end{myrem}

\subsection{Contraction Principle}  In this subsection, we shall present a contraction principle, a slight variant of Lemma 4.7 in Latała et al \cite{Latala_Inventions}. For convenience of reading, we provide a simple proof here. The reader is also referred to Ledoux and Talagrand \cite{Ledoux_Talagrand_book} for a well-known contraction principle in Banach space.

\begin{mylem}\label{Lem_Comparison}
	Let $\xi_{i}$ and $\eta_{i}$, $i=1, \cdots, n$ be independent symmetric centered random variables such that for some constant $c\ge 1$ and $t>0$
	\begin{align}
	\textsf{P}\big\{ \vert\xi_{i}\vert\ge t   \big\}\le c\textsf{P}\big\{c\vert\eta_{i}\vert\ge t  \big\}.\nonumber
	\end{align}
	Then,
	\begin{align}
		\textsf{E}f(\xi_{1},\cdots,\xi_{n})\le \textsf{E}f(c^{2}\eta_{1},\cdots,c^{2}\eta_{n} )\nonumber
	\end{align}
	for every convex function $f: \mathbb{R}^{n}\to \mathbb{R}$.
\end{mylem}

\begin{proof}
	Let $\delta_{i}\sim \text{Bern}(1/c)$ be i.i.d. Bernoulli variables independent of $\xi_{i}$. Then we have for $t\ge 0$
	\begin{align}
		\textsf{P}\{ \delta_{i}\vert \xi_{i}\vert\ge t \}\le \textsf{P}\{ c\vert \eta_{i}\vert\ge t \}.\nonumber
	\end{align}
	  We can couple $(\delta_{i}, \xi_{i})$ and $(\eta_{i})$ on the same probability space such that $\delta_{i}\vert \xi_{i}\vert\le c\vert \eta_{i}\vert$ a.s. for every $i$,  due to a standard coupling argument (see Section 1 of Chapter IV in \cite{lectures}).

	Let $\varepsilon_{i}$ be i.i.d. Rademacher variables. Note that the function $$x\mapsto \textsf{E}f(x_{1}\varepsilon_{1},\cdots,  x_{n}\varepsilon_{n})$$ is convex, so its supremum over $\prod_{i\le n}[-c\vert \eta_{i}\vert, c\vert \eta_{i}\vert ]$ is attained at one of the extreme points. We also note that
	\begin{align}
		\textsf{E}f(x_{1}\varepsilon_{1}, x_{2}\varepsilon_{2}, \cdots,  x_{n}\varepsilon_{n})=\textsf{E}f(-x_{1}\varepsilon_{1}, x_{2}\varepsilon_{2}, \cdots,  x_{n}\varepsilon_{n})\nonumber
	\end{align}
	due to that $\varepsilon_{i}$ are symmetric. Hence, we have
	\begin{align}
		\textsf{E}\Big(f\big(\varepsilon_{1}\delta_{1}\vert \xi_{1}\vert, \cdots, \varepsilon_{n}\delta_{n}\vert \xi_{n}\vert\big)\big| \delta, \xi, \eta\Big)\le \textsf{E}\Big(f\big(c\varepsilon_{1}\vert \eta_{1}\vert, \cdots, c\varepsilon_{n}\vert \eta_{n}\vert\big)\big| \delta, \xi, \eta\Big)\nonumber,
	\end{align}
	which implies that
	\begin{align}
		\textsf{E}f(\delta_{1}\xi_{1},\cdots, \delta_{n}\xi_{n})\le \textsf{E}f(c\eta_{1},\cdots, c\eta_{n}).\nonumber
	\end{align}
	Using Jensen's inequality again, we have
	\begin{align}
		\textsf{E}f(\delta_{1}\xi_{1},\cdots, \delta_{n}\xi_{n})=	\textsf{E}\textsf{E}_{\delta}f(\delta_{1}\xi_{1},\cdots, \delta_{n}\xi_{n})\ge \textsf{E}f(\xi_{1}/c,\cdots, \xi_{n}/c)\nonumber.
	\end{align}
	This concludes the proof.
\end{proof}

\subsection{The Generic Chaining}
In this subsection, we shall introduce some definitions and properties of Talagrand's celebrated generic chaining argument \cite{Talagrand_chaining_book}.

For a metric space $(T, d)$, we call a sequence of subsets $\{T_{r}: r\ge 0 \}$ of $T$ an admissible sequence if for every $r\ge 1, \vert T_{r}\vert\le 2^{2^{r}}$ and $\vert T_{0}\vert=1$. For any $0<\alpha<\infty$, the Talagrand's $\gamma_{\alpha}$-functional of $(T, d)$ is defined as follows
\begin{align}\label{Talagrand's_functional}
	\gamma_{\alpha}(T, d)=\inf \sup_{t\in T}\sum_{r\ge 0}2^{r/\alpha}d(t, T_{r}),
\end{align}
where the infimum is taken concerning all admissible sequences of $T$.  In the next subsection, we shall give an alternative definition; see (\ref{Equation_gamma_functionals}).

For a fixed radius $u>0$, we denote the covering number of $T$ by $N(T, d, u)$. One can bound the $\gamma_{\alpha}$-functionals with such covering numbers. In particular, we have
\begin{align}\label{Eq_gamma_covering}
	\gamma_{\alpha}(T, d)\lesssim_{\alpha}\int_{0}^{\infty}(\log N(T, d, u))^{1/\alpha}\, d u.
\end{align}
The bound of the above inequality is the well-known Dudley integral. One can find proof in \cite{Talagrand_chaining_book} for the case $\alpha=2$; the other cases are similar. We do not cover the proof here.

The following lemma bounds the covering number of the Euclidean ball, $B_{2}^{n}=\{x\in \mathbb{R}^{n}: \Vert x\Vert_{2}\le 1  \}$.
\begin{mylem}[Corollary 4.2.13 in \cite{highdimension}]\label{Lem_ball_estimate}
	The covering number of the unit Euclidean ball $B^{n}_{2}$ satisfy the following for any $u>0$:
	\begin{align}
		\big(\frac{1}{u}\big)^{n}\le N\big(B_{2}^{n}, \Vert\cdot\Vert_{2}, u\big)\le \big(\frac{2}{u}+1\big)^{n}.\nonumber
	\end{align}
The same upper bound is true for the unit Euclidean sphere $S^{n-1}$.
\end{mylem}

When $T$ is uncountable, we can not ensure the quantity $\sup_{t\in T}X_{t}$ is measurable. We, by convention, use the following definition for its expectation 
\begin{align}
	\textsf{E}\sup_{t\in T}X_{t}=\sup\{ \textsf{E}\sup_{t\in F}X_{t}:  F\subset T, \vert F\vert<\infty \}.\nonumber
\end{align}

\subsection{The Majorizing Measure Theorems}

With the generic chaining argument in hand, one can derive the bound of the suprema of the process by Talagrand's $\gamma$-functional. Talagrand proved $\gamma_{2}(T, d)$ is optimal for Gaussian processes with canonical distance on $T$, the majorizing measure theorem. However, this will no longer be the case in general. For some nongaussian processes, the generic chaining argument leads to a terrible upper bound; see an example in \cite{Bogucki}. Fortunately, Talagrand \cite{Talagrand_ajm_prob} and Latała \cite{Latala_gafa} generalize the majorizing measure theorem to some special log-concave-tailed canonical processes. Not long ago, Latała and Tkocz \cite{Latala_EJP} showed an optimal bound for the suprema of the canonical processes based on independent stardard variables with moments growing regularly, which is one of the most general results in this field. The chaining argument to estimate the upper bound of the suprema of processes in \cite{Latala_EJP} also appeared
in  Mendelson and Paouris's work \cite{Mendelson_JFA}. In this subsection, we mainly collect some majorizing measure theorems of a particular type of canonical processes, $\sum t_{i}\xi_{i}$, where $\sum t_{i}^{2}<\infty$ and $\{ \xi_{i}, i\ge 1\}$ are independent symmetric variables with logarithmically concave tails.

Recall that a random variable $\xi_{i}$ with logarithmically concave tails means that the function
\begin{align}
	U_{i}(x)=-\log\textsf{P}\{\vert \xi_{i}\vert \ge x \}\nonumber
\end{align}
is convex. Since it is only a matter of normalization, we assume that $U_{i}(1)=1$. Consider the function $\hat{U}_{i}(x)$ given by 
\begin{equation*}
	\hat{U}_{i}(x)=\left\{
	\begin{array}{cl}
		x^{2} &  \vert x\vert\le 1 \\
		2U(\vert x\vert)-1  &  \vert x\vert\ge 1. \\
	\end{array} \right.
\end{equation*}
Given $u>0$, we define for $\sum t_{i}^{2}\le 1$
\begin{align}
	\mathcal{N}_{u}(t)=\sup\big\{\sum_{i\ge 1}t_{i}a_{i}: \sum_{i\ge 1}\hat{U}_{i}(a_{i})\le u   \big\}.\nonumber
\end{align}
To get a feeling of this quantity, we carry out the meaning of $\mathcal{N}_{u}(t)$ in the simplest case, $U_{i}(x)=x^{2}, i\ge 1$. Note that $x^{2}\le \hat{U}_{i}(x)\le 2x^{2}$ in this case. Hence we have
\begin{align}
	\sqrt{u/2}\Vert t\Vert_{2}\le \mathcal{N}_{u}(t)\le \sqrt{u}\Vert t\Vert_{2}.\nonumber
\end{align} 
Given a set $T$, we call a sequence of partitions $(\mathcal{T}_{n})_{n\ge 0}$ of $T$ are admissible if $\vert \mathcal{T}_{0}\vert=1$, $\vert \mathcal{T}_{n}\vert\le 2^{2^{n}}$ for $n\ge 1$ and every set of $\mathcal{T}_{n+1}$ is contained in a set of $\mathcal{T}_{n}$. We denote by $T_{n}(t)$ the unique element of $\mathcal{T}_{n}$ which contains $t$. We denote by $\Delta_{d}(T)$ the diameter of the set $T$. Let
\begin{align}\label{Equation_gamma_functionals}
	\gamma_{\alpha}^{\prime}(T, d)=\inf_{\mathcal{T}}\sup_{t\in T}\sum_{n\ge 0}2^{n/\alpha}\Delta_{d}(T_{n}(t))
\end{align}
where the infimum is taken over all admissible partitions $\mathcal{T}$ of $T$. Talagrand \cite{Talagrand_annals_prob} showed that
\begin{align}
	\gamma_{\alpha}(T, d)\le \gamma_{\alpha}^{\prime}(T, d)\le C(\alpha)\gamma_{\alpha}(T, d).\nonumber
\end{align}
Define $B(u)=\{t: \mathcal{N}_{u}(t)\le u \}$. Given a number $r\ge 4$, let
\begin{align}
	\phi_{j}(s, t)=\inf\{u>0: s-t\in r^{-j}B(u)  \}.\nonumber
\end{align}

\begin{mylem}[Theorem 8.3.2 in \cite{Talagrand_chaining_book}]\label{Lemma_chaining}
	Let $\{\xi_{i}, i\ge 1\}$ be a sequence of independent symmetric variables with log-concave tails and $U_{i}(x)=-\log \textsf{P}\{\vert\xi_{i}\vert\ge x\}$. Assume there is a finite positive constant $\beta_{0}$ such that for any $i\ge 1, x\ge 1$
	\begin{align}
		 U_{i}(2x)\le \beta_{0}U_{i}(x), \quad U_{i}^{'}(0)\ge \frac{1}{\beta_{0}}.\nonumber
	\end{align}
Then, there exists $r_{0}(\beta_{0})$ and $K(\beta_{0})$ such that when $r\ge r_{0}(\beta_{0})$ and for each subset $T$ of $\{t :\sum t_{i}^{2}<\infty\}$ we can find an admissible sequence of partitions $(\mathcal{T}_{n})_{n\ge 0}$ of $T$ and for $T_{n}\in\mathcal{T}_{n}$ an integer $j_{n}(T_{n})\in \mathbb{Z}$ satisfying
\begin{align}
	\forall t, t^{\prime}\in T_{n},\quad \phi_{j_{n}(T_{n})}(t, t^{\prime})\le 2^{n}\nonumber
\end{align}
and 
\begin{align}
	\frac{1}{K(\beta_{0})r}\sup_{t\in T}\sum_{n\ge 0}2^{n}r^{-j_{n}(T_{n}(t))}\le \textsf{E}\sup_{t\in T}\sum t_{i}\xi_{i}\le C\sup_{t\in T}\sum_{n\ge 0}2^{n}r^{-j_{n}(T_{n}(t))},\nonumber
\end{align}
where $C$ is a universal constant and $T_{n}(t)$ is the unique set containing $t$ from $\mathcal{T}_{n}$.
\end{mylem}

\begin{myrem}\label{Rem_gamma_equavilant}
	Let $\Delta_{2^{n}}(\cdot)$ be the diameter with respect to the distance induced by $L_{2^{n}}$ norm. In particular, $\Delta_{2^{n}}(T)=\sup_{t, s\in T}\big\Vert \sum (t_{i}-s_{i})\xi_{i} \big\Vert_{L_{2^{n}}}$Then,
	\begin{align}
		\inf\sup_{t\in T}\sum_{n\ge 0}\Delta_{2^{n}}(T_{n}(t))\asymp_{\beta}\sup_{t\in T}\sum_{n\ge 0}2^{n}r^{-j_{n}(T_{n}(t))},\nonumber
	\end{align}
where the infimum runs over all admissible sequences of partitions of $T$. Interested readers can refer to Remark 1.6 in \cite{Latala_EJP} for a detailed explanation.
\end{myrem}

In the special cases $U_{i}(x)=x^{\alpha}, 1\le \alpha\le 2$,  Talagrand \cite{Talagrand_ajm_prob} reformulates Lemma \ref{Lemma_chaining} in the following familiar version.

\begin{mylem}[Theorem 1.2 in \cite{Talagrand_ajm_prob}]\label{Lemma_chaining2}
	Let $T\subset \{t :\sum t_{i}^{2}<\infty\}$ and $\{\xi_{n}, n\ge 1 \}$ be a sequence of independent symmetric variables with $U_{i}(x)=x^{\alpha}, 1\le \alpha\le 2$. Set $\alpha^{*}=\alpha/(\alpha -1)$ be the conjugate exponent of $\alpha$. Then
	\begin{align}
		 \textsf{E}\sup_{t\in T}\sum t_{i}\xi_{i}\asymp_{\alpha} \max\big(\gamma_{2}(T, \Vert\cdot\Vert_{2}),\gamma_{\alpha}(T, \Vert\cdot\Vert_{\alpha^{*}})  \big),\nonumber
	\end{align}
where $(T, \Vert\cdot\Vert_{\alpha^{*}})$ means we use the distance induced by the norm $\Vert\cdot\Vert_{\alpha^{*}}$ on $T$.
\end{mylem}

\begin{myrem}\label{Rem_majorizing_measure}
	(i) Readers might notice that the bound in \cite{Talagrand_ajm_prob} contains the following quantity:
	\begin{align}
		\inf \sup_{t\in T}\int_{0}^{\infty} \log^{1/\alpha}\big(\frac{1}{\mu\{B_{\alpha^{*}}(u, \varepsilon) \}}\big)\, d\varepsilon,\nonumber
	\end{align}
	where $B_{\alpha^{*}}(u, \varepsilon)=\big\{ t\in T: \Vert t-u \Vert_{\alpha^{*}}\le \varepsilon  \big\}$, and the infimum is taken over of the probability measures $\mu$ on $T$. Talagrand \cite{Talagrand_annals_prob} showed that this quantity is equivalent to $\gamma_{\alpha}(T, \Vert\cdot\Vert_{\alpha^{*}})$. We use such version here due to that  $\gamma_{\alpha}(T, \Vert\cdot\Vert_{\alpha^{*}})$ is easy to obtain by a chaining argument.

	(ii) One can deduce from Lemma \ref{Lemma_chaining} and Lemma \ref{Lemma_chaining2} that,
	\begin{align}
		\gamma_{2}(T, \Vert\cdot\Vert_{2})+\gamma_{\alpha}(T, \Vert\cdot\Vert_{\alpha^{*}})\asymp_{\alpha, \beta_{0}} \sup_{t\in T}\sum_{n\ge 0}2^{n}r^{-j_{n}(A_{n}(t))}.\nonumber
	\end{align}
 Talagrand showed the equivalence of the above two quantities for the cases $\alpha=1, 2$; see Section 8.3 in \cite{Talagrand_chaining_book}. The other cases are similar.
\end{myrem}

\section{Proofs for Our Main Results}

\subsection{The Proof of Theorem \ref{Theo_logconcave_deviation}}

Recall that a homogeneous chaos is a random variable with the following form
\begin{align}
	S_A=\sum_{i, j=1}^{n}a_{ij}\xi_{i} \xi_{j},\nonumber
\end{align}
where $\xi_{1},\cdots, \xi_{n}$ are independent centered variables and $A=(a_{ij})$ is a fixed matrix. $S_A$ is tetrahedral if all diagonal entries of $(a_{ij})$ are zero.

As a basic tool for studying chaoses, the decoupling technique is applied to control $S_A$ by its decoupled version $\tilde{S}_{A}$ of the following form
\begin{align}
	\tilde{S}_A:=\sum_{i, j=1}^{n}a_{ij}\xi_{i}\tilde{\xi}_{j},\nonumber
\end{align}
where $(\tilde{\xi}_{1},\cdots, \tilde{\xi}_{n})^\top$ are independent copies of $(\xi_{1},\cdots, \xi_{n})^\top$.
 Kwapie\'{n} \cite{Kwapien_Decoupling} obtained the following  classic decoupling result for the tetrahedral case,
\begin{align}\label{Res_decoup}
	\textsf{E}F\big(\sum_{i\neq j}a_{ij}\xi_{i}\xi_{j}\big)\le \textsf{E}F\big(4\sum_{i, j}a_{ij}\xi_{i}\tilde{\xi_{j}}\big),
\end{align}
where $(\tilde{\xi_{i}}, i\le n)$ is an independent copy of $(\xi_{i}, i\le n)$ and $F: \mathbb{R}\to \mathbb{R}$ is some convex function. Let $\mathcal{A}$ be a collection of matrices. The following decoupling inequality is a slight variation of (\ref{Res_decoup}), see Theorem 2.4 in \cite{Rauhut_CPAM}.
\begin{align}\label{Res_decoup_variation}
	\textsf{E}\sup_{A\in\mathcal{A}}F\big(\sum_{i\neq j}a_{ij}\xi_{i}\xi_{j}\big)\le \textsf{E}\sup_{A\in\mathcal{A}}F\big(4\sum_{i, j}a_{ij}\xi_{i}\tilde{\xi_{j}}\big).
\end{align}

Arcones and Gin\'{e} \cite{Arcones_Gine} obtained a decoupling inequality for nonhomogeneous Gaussian chaoses generated by \eqref{Eq_polynomial}. In particular, let $g=(g_{1},\cdots, g_{n})^\top$ be a standard Gaussian vector ($g_{i}, i=1,\cdots, n$ are i.i.d. standard Gaussian variables) and $\tilde{g}$ be an independent copy of $g$, then
\begin{align}\label{Gaussian_decoupling}
	\textsf{E}\sup_{A\in \mathcal{A}}\vert g^\top Ag-\textsf{E} g^\top Ag \vert^{p}\le C^{p}\textsf{E}\sup_{A\in\mathcal{A}}\vert g^\top A\tilde{g}\vert^{p}.
\end{align}

To prove Theorem \ref{Theo_logconcave_deviation}, we first show the following novel decoupling inequality.

\begin{mypro}\label{Prop_decoupling}
	Let $F$ be a convex function satisfying $F(x)=F(-x), \forall x\in \mathbb{R}$.  Assume $\xi=(\xi_{1}, \cdots, \xi_{n})^\top$ and $ \eta=(\eta_{1}, \cdots, \eta_{n})^\top$ are random vectors with independent centered entries. Let for any $t>0$ and $i\ge 1$, the independent random variables $\xi_{i}, \eta_{i}$ satisfy for  some $c\ge 1$
	\begin{align}\label{Condition_tail_decoupling}
		\textsf{P}\{ \xi_{i}^{2}\ge t \}\le c\textsf{P}\{ c\vert \eta_{i}\tilde{\eta}_{i}\vert\ge t \},
	\end{align}
	where $\tilde{\eta}_{i}$ is an independent copy of $\eta_{i}$. 
	Then, there exists a constant $C$ depending only on $c$ such that
	\begin{align}
		\textsf{E}\sup_{A\in \mathcal{A}}F\big(\xi^\top A\xi-\textsf{E}\xi^\top A\xi\big)\le \textsf{E}\sup_{A\in \mathcal{A}}F\big( C\eta^\top A\tilde{\eta}\big)\nonumber
	\end{align}
	where  $\mathcal{A}$ is a set of $n\times n$ fixed matrices.
\end{mypro}

\begin{myrem}\label{Rem_decoupling_inequality_2}
	(i)  Let the entries of $\xi$ be independent centered variables with  log-concave tails, namely the functions
	\begin{align}
		U_{i}(t):=-\log \textsf{P}\{\vert \xi_{i}\vert\ge t  \}, \quad i\le n\nonumber
	\end{align}
	are convex. Set $\eta, \tilde{\eta}$ be independent copies of $\xi$. Note that $U_{i}(0)=0$ and $U_{i}^{\prime}(0)\ge 0$ for $i\le n$. Hence, $\xi, \eta$ and $\tilde{\eta}$ satisfy the condition (\ref{Condition_tail_decoupling}).
	
	(ii) Let $\xi, \eta$ and $\tilde{\eta}$ be independent standard Gaussian vectors. Note that there exists a constant $C$ such that
	\begin{align}
		\frac{\sqrt{p}}{C}\le (\textsf{E}\vert\xi_{i}\vert^{p})^{1/p}\le C\sqrt{p}, \quad 1\le p,\,\, 1\le i\le n.\nonumber 
	\end{align}
	Then, $\xi, \eta$ and $\tilde{\eta}$ satisfy the condition (\ref{Condition_tail_decoupling}) due to Lemma \ref{Lem_low bound}.	
\end{myrem}

\begin{proof}[Proof of Proposition \ref{Prop_decoupling}]
	Without loss of generality, we assume $\textsf{E}\xi_{i}^{2}=1$. Due to the convexity of $F$, we have
	\begin{align}\label{Eq_decoupling_decomposition_1}
		\textsf{E}\sup_{A\in \mathcal{A}}F\big(\xi^\top A\xi -\textsf{E}\xi^\top A\xi\big)&=\textsf{E}\sup_{A\in \mathcal{A}}F\Big(\sum_{i\neq j}a_{ij}\xi_{i}\xi_{j}+\sum_{i\le n}(\xi_{i}^{2}-1)a_{ii}\Big)\nonumber\\
		&\le \frac{1}{2}\textsf{E}\sup_{A\in \mathcal{A}}F\big(2\sum_{i\neq j}a_{ij}\xi_{i}\xi_{j}\big)+\frac{1}{2}\textsf{E}\sup_{A\in \mathcal{A}}F\big(2\sum_{i\le n}(\xi_{i}^{2}-1)a_{ii}\big).
	\end{align}
By the condition \eqref{Condition_tail_decoupling}, we have for $t\ge 0$
\begin{align}
	\textsf{P}\{\vert\xi_{i}\vert>t  \}&=\textsf{P}\{\xi_{i}^{2}>t^{2}  \}\le c\textsf{P}\{ c\vert \eta_{i}\tilde{\eta}_{i}\vert\ge t^{2} \}\nonumber\\
	&\le c\big( \textsf{P}\{\sqrt{c}\vert\eta_{i}\vert>t  \}  +\textsf{P}\{\sqrt{c}\vert\tilde{\eta}_{i}\vert>t  \} \big)\le 2c\textsf{P}\{\sqrt{c}\vert\eta_{i}\vert>t  \}.\nonumber
\end{align}
Hence, the classical decoupling inequality (\ref{Res_decoup_variation}) and Lemma \ref{Lem_Comparison}  lead to that
\begin{align}
	\textsf{E}\sup_{A\in \mathcal{A}}F(2\sum_{i\neq j}a_{ij}\xi_{i}\xi_{j})&\le \textsf{E}\sup_{A\in \mathcal{A}}F(8\sum_{i, j}a_{ij}\xi_{i}\tilde{\xi}_{j})\le\textsf{E}\sup_{A\in \mathcal{A}}F\big(32  \sum_{i, j}a_{ij} (\varepsilon_{i}\xi_{i})(\tilde{\varepsilon}_{j}\tilde{\xi}_{j})\big) \nonumber\\
	&\le \textsf{E}\sup_{A\in \mathcal{A}}F\big(576c^{4}  \sum_{i, j}a_{ij} (\epsilon_{i}\eta_{i})(\tilde{\epsilon}_{j}\tilde{\eta}_{j})\big)\nonumber\\
	&\le \textsf{E}\sup_{A\in \mathcal{A}}F\big(2304c^{4}  \sum_{i, j}a_{ij} \eta_{i}\tilde{\eta}_{j}\big).\nonumber
\end{align}
Here, $\{\varepsilon_{i}, \epsilon_{i}, i\le n  \}$ is a sequence of i.i.d. Rademacher variables independent of the other variables,  and the second and the fourth inequalities are due to standard symmetrization and  desymmetrization arguments; see Lemma 6.3 in \cite{Ledoux_Talagrand_book}.

We next turn to the other summand term in (\ref{Eq_decoupling_decomposition_1}).
\begin{align}
	\textsf{E}\sup_{A\in \mathcal{A}}F\big(2\sum_{i\le n}(\xi_{i}^{2}-1)a_{ii}\big)&=\textsf{E}\sup_{A\in \mathcal{A}}F\big(2\sum_{i\le n}(\xi_{i}^{2}-\textsf{E}_{\tilde{\xi}}\tilde{\xi}_{i}^{2})a_{ii}\big)\nonumber\\
	&\le \textsf{E}\textsf{E}_{\tilde{\xi}}\sup_{A\in \mathcal{A}}F\big(2\sum_{i\le n}(\xi_{i}^{2}-\tilde{\xi}^{2}_{i})a_{ii}\big)
	=\textsf{E}\sup_{A\in \mathcal{A}}F\big(2\sum_{i\le n}\varepsilon_{i}(\xi_{i}^{2}-\tilde{\xi}^{2}_{i})a_{ii}\big)\nonumber\\
	&\le \frac{1}{2}\textsf{E}\sup_{A\in \mathcal{A}}F(4\sum_{i\le n}\varepsilon_{i}\xi_{i}^{2}a_{ii})+\frac{1}{2}\textsf{E}\sup_{A\in \mathcal{A}}F(-4\sum_{i\le n}\varepsilon_{i}\tilde{\xi}^{2}a_{ii})\nonumber\\
	&\le \textsf{E}\sup_{A\in \mathcal{A}}F(4\sum_{i\le n}\varepsilon_{i}\xi_{i}^{2}a_{ii}),\nonumber
\end{align}
where  $\textsf{E}_{\tilde{\xi}}$ means that we take expectations with respect to $\tilde{\xi}$, i.e., conditioned on $\xi$. 

It follows from Lemma \ref{Lem_Comparison} and the condition (\ref{Condition_tail_decoupling}) that
\begin{align}
	\textsf{E}\sup_{A\in \mathcal{A}}F\big(4\sum_{i\le n}\varepsilon_{i}\xi_{i}^{2}a_{ii}\big)&\le \textsf{E}\sup_{A\in \mathcal{A}}F\big(4c^{2}\sum_{i\le n}\varepsilon_{i}\eta_{i}\tilde{\eta}_{i}a_{ii}\big)\nonumber\\
	&\le \textsf{E}\sup_{A\in \mathcal{A}}F\big(8c^{2}\sum_{i\le n}\eta_{i}\tilde{\eta}_{i}a_{ii}\big).
\end{align}

We have, by the convexity and symmetry of $F$,
\begin{align}
	\textsf{E}\sup_{A\in \mathcal{A}}F\big(8c^{2}\sum_{i\le n}a_{ii}\eta_{i}\tilde{\eta}_{i}\big)&=\textsf{E}\sup_{A\in \mathcal{A}}F\big(8c^{2}\sum_{i,  j}a_{ij}\eta_{i}\tilde{\eta}_{j}-8c^{2}\sum_{i\neq  j}a_{ij}\eta_{i}\tilde{\eta}_{j}\big)\nonumber\\
	&\le \frac{1}{2}\textsf{E}\sup_{A\in \mathcal{A}}F\big(16c^{2}\sum_{i,  j}a_{ij}\eta_{i}\tilde{\eta}_{j}\big)+\frac{1}{2}\textsf{E}\sup_{A\in \mathcal{A}}F\big(16c^{2}\sum_{i\neq  j}a_{ij}\eta_{i}\tilde{\eta}_{j}\big).\nonumber
\end{align}

Next, we shall bound $\textsf{E}\sup_{A\in \mathcal{A}}F\big(16c^{2}\sum_{i\neq  j}a_{ij}\eta_{i}\tilde{\eta}_{j}\big)$ with a classical argument. Consider a sequence of independent Bernoulli variables $\delta_{1}, \cdots, \delta_{n}$ with $\textsf{P}\{ \delta_{i}=0\}=\textsf{P}\{ \delta_{i}=1\}=1/2$. Define a random index set as follows
\begin{align}
	I:=\{ i: \delta_{i} =1\}.\nonumber
\end{align}

Note that
\begin{align}
	\sum_{i\neq  j}a_{ij}\eta_{i}\tilde{\eta}_{j}=4\textsf{E}_{\delta}\sum_{i\neq  j}\delta_{i}(1-\delta_{j})a_{ij}\eta_{i}\tilde{\eta}_{j}=4\textsf{E}_{\delta}\sum_{(i, j)\in I\times I^{c}}a_{ij}\eta_{i}\tilde{\eta}_{j}.\nonumber
\end{align}
Hence, the convexity of $F$ yields that
\begin{align}
	\textsf{E}\sup_{A\in \mathcal{A}}F\big(16c^{2}\sum_{i\neq  j}a_{ij}\eta_{i}\tilde{\eta}_{j}\big)&=\textsf{E}\sup_{A\in \mathcal{A}}F\big(64c^{2}\textsf{E}_{\delta}\sum_{(i, j)\in I\times I^{c}}a_{ij}\eta_{i}\tilde{\eta}_{j}\big)\nonumber\\
	&\le \textsf{E}\sup_{A\in \mathcal{A}}F\big(64c^{2}\sum_{(i, j)\in I\times I^{c}}a_{ij}\eta_{i}\tilde{\eta}_{j}\big).\nonumber
\end{align}

Let $[n]=\{1,\cdots, n\}$. We make the following decomposition
\begin{align}
	\sum_{(i, j)\in [n]\times [n]}a_{ij}\eta_{i}\tilde{\eta}_{j}&=\sum_{(i, j)\in I\times I^{c}}a_{ij}\eta_{i}\tilde{\eta}_{j}+\sum_{(i, j)\in I\times I}a_{ij}\eta_{i}\tilde{\eta}_{j}+\sum_{(i, j)\in  I^{c}\times [n]}a_{ij}\eta_{i}\tilde{\eta}_{j}\nonumber\\
	&=: Y+Z_{1}+Z_{2}.\nonumber
\end{align}
Condition on all random variables except $(\eta_{i})_{i\in I^{c}}$ and $(\tilde{\eta}_{j})_{j\in I}$. We denote this conditional expectation by $\textsf{E}^{\prime}$. Note that
\begin{align}
	\sup_{A\in \mathcal{A}}F(64c^{2} Y)&=\sup_{A\in \mathcal{A}}F\Big(64c^{2} \big(Y+\textsf{E}^{\prime}(Z_{1}+Z_{2})\big)\Big)\nonumber\\
	&\le \textsf{E}^{\prime}\sup_{A\in \mathcal{A}}F\big(64c^{2}(Y+Z_{1}+Z_{2})\big).\nonumber
\end{align} 
Taking the expectations of both sides implies that
\begin{align}
	\textsf{E}\sup_{A\in \mathcal{A}}F\big(64c^{2}\sum_{(i, j)\in I\times I^{c}}a_{ij}\eta_{i}\tilde{\eta}_{j}\big)\le \textsf{E}\sup_{A\in \mathcal{A}}F\big(64c^{2}\sum_{(i, j)\in [n]\times [n]}a_{ij}\eta_{i}\tilde{\eta}_{j}\big).
\end{align}

Until now, we have obtained 
\begin{align}
	\textsf{E}\sup_{A\in \mathcal{A}}F(\xi^\top A\xi-\textsf{E}\xi^\top A\xi)\le& \frac{1}{2}\textsf{E}\sup_{A\in \mathcal{A}}F(2304c^{4}\eta^\top A\tilde{\eta})\nonumber\\
	&+\frac{1}{4}\textsf{E}\sup_{A\in \mathcal{A}}F(16c^{2}\eta^\top A\tilde{\eta})+\frac{1}{4}\textsf{E}\sup_{A\in \mathcal{A}}F(64c^{2}\eta^\top A\tilde{\eta}).\nonumber
\end{align}
Let $C=2304c^{4}$. Using Lemma \ref{Lem_Comparison} again, we obtain the desired result. 
\end{proof}

With Proposition \ref{Prop_decoupling} in hand, we shall prove Theorem \ref{Theo_logconcave_deviation} via a chaining argument, appearing in \cite{Latala_EJP,Mendelson_JFA}. This chaining is based on the distance generated by $\Vert\cdot\Vert_{L_{2^{n}}}$ (see \eqref{Eq_2n} below). We also note that log-concave-tailed variables are subexponential. Hence, for $p\ge 1$, its $p$-th moment exists.

\begin{proof}[Proof of Theorem \ref{Theo_logconcave_deviation}]
	
	We define a subset $\mathcal{A}_{n}$ of $\mathcal{A}$ by selecting exactly one point from each $T\in \mathcal{T}_{n}$, where $(\mathcal{T}_{n})_{n\ge 0}$ is an admissible sequence of partitions of $\mathcal{A}$. In this way, we obtain an admissible sequence $(\mathcal{A}_{n})_{n\ge 0}$ of subsets of $\mathcal{A}$. Let $\pi=\{ \pi_{n}, n\ge 0\}$ be a sequence of functions $\pi_{n}: \mathcal{A}\to \mathcal{A}_{n}$ such that $\pi_{n}(A)=\mathcal{A}_{n}\cap T_{n}(A)$, where $T_{n}(A)$ is the element of $\mathcal{T}_{n}$ containing $A$. Let $l$ be the largest integer such that $2^{l}\le 4p$.
	
	We first make the following decomposition:
	\begin{align}\label{Eq_decomposition_1_new}
		\big\Vert \sup_{A\in \mathcal{A}}\vert \xi^\top  A\tilde{\xi}\vert\big\Vert_{L_{p}}\le\big\Vert \sup_{A\in \mathcal{A}}\vert \xi^\top  A\tilde{\xi}-\xi^\top \pi_{l}(A)\tilde{\xi}\vert\big\Vert_{L_{p}}+\big\Vert \sup_{A\in \mathcal{A}}\vert \xi^\top \pi_{l}(A)\tilde{\xi}\vert\big\Vert_{L_{p}}.
	\end{align}
	Then, we build the following chain:
	\begin{align}
		\xi^\top  A\tilde{\xi}-\xi^\top \pi_{l}(A)\tilde{\xi}
		=\sum_{n\ge l}\Big(\xi^\top \pi_{n+1}(A)\tilde{\xi}-\xi^\top \pi_{n}(A)\tilde{\xi}\Big)\nonumber.
	\end{align}
	Denote
	\begin{align}\label{Eq_2n}
		d_{2^{n}}\big(\pi_{n+1}(A), \pi_{n}(A)
		=\Big(\textsf{E}_{\xi} \big\vert\xi^\top \pi_{n+1}(A)\tilde{\xi}-\xi^\top \pi_{n}(A)\tilde{\xi}\big\vert^{2^{n}}\Big)^{1/2^{n}}
	\end{align}
	and set $\Omega_{n, t, u}$ be the event
	\begin{align}
		\Big\{ \big\vert\xi^\top \pi_{n+1}(A)\tilde{\xi}-\xi^\top \pi_{n}(A)\tilde{\xi} \big\vert\le u\cdot d_{2^{n}}\big(\pi_{n+1}(A), \pi_{n}(A)\big) \Big \}.\nonumber
	\end{align}
	Conditioned on $\tilde{\xi}$, Chebyshev's inequality yields $\textsf{P}_{\xi}\{\Omega_{n, t, u}^{c}  \}\le u^{-2^{n}}$, so if we set $\Omega(u)=\bigcap_{n\ge l}\bigcap_{A}\Omega_{n, t, u}$, by the union bound we easily find that
	\begin{align}
		\textsf{P}_{\xi}\{\Omega(u)^{c} \}\le \sum_{n\ge l}\vert \mathcal{A}_{n+1}\vert\vert\mathcal{A}_{n}\vert u^{-2^{n}}\le \sum_{n\ge l} \big(\frac{8}{u}\big)^{2^{n}}\le \frac{2\cdot 64^{p}}{u^{2p}},\quad
		u\ge 16.\nonumber
	\end{align}
	Since on $\Omega_{u}$ we have
	\begin{align}
		\sup_{A\in \mathcal{A}} \vert\xi^\top  A\tilde{\xi}-\xi^\top \pi_{l}(A)\tilde{\xi} \vert\le u\cdot\sup_{A\in\mathcal{A}} \sum_{n\ge l}d_{2^{n}}\big(\pi_{n+1}(A), \pi_{n}(A)\big)\nonumber.
	\end{align}
	Then, we have for $u\ge 16$
	\begin{align}
		\textsf{P}_{\xi}\Big\{ &\sup_{A\in \mathcal{A}} \vert\xi^\top  A\tilde{\xi}-\xi^\top \pi_{l}(A)\tilde{\xi} \vert> u\cdot\sup_{A\in\mathcal{A}} \sum_{n\ge l}d_{2^{n}}\big(\pi_{n+1}(A), \pi_{n}(A)\big)  \Big\}\le \frac{2\cdot 64^{p}}{u^{2p}}.\nonumber
	\end{align}
	Integration yields that
	\begin{align}
		\Big(\textsf{E}_{\xi} \sup_{A\in \mathcal{A}}\vert \xi^\top  A\tilde{\xi}-\xi^\top \pi_{l}(A)\tilde{\xi}\vert^{p}\Big)^{1/p}&\lesssim \sup_{A\in\mathcal{A}} \sum_{n\ge l}d_{2^{n}}\big(\pi_{n+1}(A), \pi_{n}(A)\big)\nonumber\\
		&\lesssim \sup_{A\in\mathcal{A}} \sum_{n\ge 0}\Delta_{2^{n}}\big(T_{n}(A) \big),\nonumber
	\end{align}
	where $\Delta_{2^{n}}\big(T_{n}(A) \big)$ is the diameter of the unique set $T_{n}(A)$ from $\mathcal{T}_{n}$ containing $A$ with respect to the distance $d_{2^{n}}(\cdot, \cdot)$. Then, 
	\begin{align}
		\Big(\textsf{E}_{\xi} \sup_{A\in \mathcal{A}}\vert \xi^\top  A\tilde{\xi}-\xi^\top \pi_{l}(A)\tilde{\xi}\vert^{p}\Big)^{1/p}\lesssim \inf\sup_{A\in\mathcal{A}} \sum_{n\ge 0}\Delta_{2^{n}}\big(T_{n}(A) \big),\nonumber
	\end{align}
	where the infimum runs over all admissible sequences of partitions $(\mathcal{T}_{n})$ of the set $\mathcal{A}$.
	Set $\varepsilon=(\varepsilon_{1},\cdots, \varepsilon_{n})^\top$ independent of $\xi, \tilde{\xi}$. Then, it follows from Lemma \ref{Lemma_chaining} and Remark \ref{Rem_gamma_equavilant}
	\begin{align}
		\inf\sup_{A\in\mathcal{A}} \sum_{n\ge l}\Delta_{2^{n}}\big(T_{n}(A) \big)&\lesssim_{\beta_{0}} \textsf{E}_{\xi, \varepsilon} \sup_{A\in \mathcal{A}}\big( (\varepsilon\circ\xi)^\top  A\tilde{\xi}\big)\lesssim \textsf{E}_{\xi} \sup_{A\in \mathcal{A}}\big\vert \xi^\top  A\tilde{\xi}\big\vert,
	\end{align}
	where $\varepsilon\circ\xi=(\varepsilon_{i}\xi_{i}, i\le n)^\top$ and the inequalities are due to the symmetrization argument (see Lemma 6.3 in \cite{Ledoux_Talagrand_book}).
	
	As for the other part of (\ref{Eq_decomposition_1_new}), we have for $p\ge 1$
	\begin{align}
		\textsf{E}\sup_{A\in \mathcal{A}}\vert \xi^\top \pi_{l}(A)\tilde{\xi}\vert^{p}\le \sum_{B\in \mathcal{A}_{l}}	\textsf{E}\vert \xi^\top  B\tilde{\xi}\vert^{p}\le 16^{p}\sup_{A\in \mathcal{A}}\textsf{E}\vert \xi^\top  A\tilde{\xi}\vert^{p},\nonumber
	\end{align}
	and thus
	\begin{align}
		\Big\Vert\sup_{A\in \mathcal{A}}\vert \xi^\top \pi_{l}(A)\tilde{\xi}\vert\Big\Vert_{L_{p}}\le 16\sup_{A\in \mathcal{A}}\Vert \xi^\top  A\tilde{\xi}\Vert_{L_{p}}.\nonumber
	\end{align}
	Then, Remark \ref{Rem_decoupling_inequality_2} (i) yields that
	\begin{align}
		\Big\Vert \sup_{A\in \mathcal{A}}\big\vert S_{\mathcal{A}}(\xi)-\textsf{E}S_{\mathcal{A}}(\xi)\big\vert\Big\Vert_{L_{p}}\lesssim_{\beta_{0}}\Big\Vert \textsf{E}_{\xi}\sup_{A\in \mathcal{A}}\big\vert \xi^\top  A\tilde{\xi}\big\vert \Big\Vert_{L_{p}}+\sup_{A\in\mathcal{A}}\big\Vert  \xi^\top  A\tilde{\xi}\big\Vert_{L_{p}},\nonumber
	\end{align}	
		which concludes the proof.
\end{proof}

\subsection{Proofs for Corollaries}

 \begin{proof}[Proof of Corollary \ref{Cor_logconcave}]
 	Let $\{\zeta_{1}, \cdots, \zeta_{n}\}\stackrel{\text{i.i.d}}{\sim}\mathcal{W}_{s}(\alpha)$. It follows from  the proof of Theorem \ref{Theo_logconcave_deviation}
 	\begin{align} 
 		\Big\Vert \sup_{A\in \mathcal{A}}\big\vert S_{\mathcal{A}}(\xi)-\textsf{E}S_{\mathcal{A}}(\xi)\big\vert\Big\Vert_{L_{p}}&\lesssim_{\alpha,L}\Big\Vert \sup_{A\in \mathcal{A}}\big\vert \zeta^\top  A\tilde{\zeta}\big\vert\Big\Vert_{L_{p}}\nonumber\\
 		&\lesssim_{\alpha,L}\Big\Vert \textsf{E}_{\zeta}\sup_{A\in \mathcal{A}}\big\vert \zeta^\top  A\tilde{\zeta}\big\vert \Big\Vert_{L_{p}}+\sup_{A\in\mathcal{A}}\big\Vert  \zeta^\top  A\tilde{\zeta}\big\Vert_{L_{p}},\nonumber
 	\end{align}
where $\tilde{\zeta}$ is an independent copy of $\zeta=(\zeta_{1}, \cdots, \zeta_{n})$.
 
 Note that,  Remark \ref{Rem_explanations} implies
 \begin{align}
 	\sup_{A\in\mathcal{A}}\big\Vert  \zeta^\top  A\tilde{\zeta}\big\Vert_{L_{p}}\lesssim_{\alpha} p^{1/\alpha} \sup_{A\in \mathcal{A}}\Vert  A\Vert_{F}+p^{2/\alpha} \sup_{A\in \mathcal{A}}\Vert  A\Vert_{l_{2}\to l_{2}}.\nonumber
 \end{align}

For the other part,  Lemma \ref{Lem_Comparison}  leads to that 
\begin{align}
	\Big\Vert \textsf{E}_{\eta}\sup_{A\in \mathcal{A}}\big\vert \zeta^\top  A\tilde{\zeta}\big\vert \Big\Vert_{L_{p}}\asymp\Big\Vert \textsf{E}_{\eta}\sup_{A\in \mathcal{A}}\big\vert \eta^\top  A\tilde{\eta}\big\vert \Big\Vert_{L_{p}},\nonumber
\end{align}
where $\eta, \tilde{\eta}$ are independent random vectors mentioned in conditions.

For convenience, let $f(x)=\textsf{E}_{\eta}\sup_{A\in \mathcal{A}}\big\vert \eta^\top  Ax\big\vert$, $x\in\mathbb{R}^{n}$. It is obvious for any $x, y\in \mathbb{R}^{n}, 1\le \alpha\le 2$ that
\begin{align}
	\vert f(x)-f(y)\vert\le \big(\textsf{E}\sup_{A\in \mathcal{A}}\Vert  A\eta\Vert_{\alpha^{*}}\big)\Vert x-y\Vert_{\alpha}\nonumber.
\end{align}
Hence, the classic concentration inequalities (see Theorem 4.19 in \cite{Ledoux_book}) yield for $t\ge 0$
\begin{align}
	\textsf{P}\{ f(\tilde{\eta})-\textsf{E}f(\tilde{\eta}) \ge t\}\le \exp\Big(  -c\min\Big\{\frac{t^{2}}{\big(\textsf{E}\sup_{A\in \mathcal{A}}\Vert  A\eta\Vert_{2}\big)^{2}}, \frac{t^{\alpha}}{\big(\textsf{E}\sup_{A\in \mathcal{A}}\Vert  A\eta\Vert_{\alpha^{*}}\big)^{\alpha}}  \Big\}\Big).\nonumber
\end{align}
Integration yields that
\begin{align}
	\Vert f(\tilde{\eta}) \Vert_{L_{p}} \lesssim \textsf{E}f(\tilde{\eta})+\sqrt{p}\textsf{E}\sup_{A\in \mathcal{A}}\Vert  A\eta\Vert_{2}+p^{1/\alpha}\textsf{E}\sup_{A\in \mathcal{A}}\Vert  A\eta\Vert_{\alpha^{*}}.\nonumber
\end{align}
Hence, we have
\begin{align}
	\Big\Vert \sup_{A\in \mathcal{A}}\big\vert S_{\mathcal{A}}(\xi)-\textsf{E}S_{\mathcal{A}}(\xi)\big\vert\Big\Vert_{L_{p}}\lesssim_{\alpha, L}& \textsf{E}\sup_{A\in \mathcal{A}}\big\vert \eta^\top  A\tilde{\eta}\big\vert+p^{2/\alpha} \sup_{A\in \mathcal{A}}\Vert  A\Vert_{l_{2}\to l_{2}}\nonumber\\
	&+\sqrt{p} \max\Big\{\textsf{E}\sup_{A\in \mathcal{A}}\Vert  A\eta\Vert_{2}, \sup_{A\in \mathcal{A}}\Vert A\Vert_{F}\Big\}\nonumber\\
	&+p^{1/\alpha}\textsf{E}\sup_{A\in \mathcal{A}}\Vert  A\eta\Vert_{\alpha^{*}}.\nonumber
\end{align}
With Lemma \ref{Lem_Moments}, we can obtain the corresponding deviation inequality.
 \end{proof}

 We next turn to prove Corollary \ref{Theo_deviation}. To this end, we need the following lemma.

\begin{mylem}\label{Theo_bound_decoupled}
	Let $\zeta_{1},\cdots,  \zeta_{n}\stackrel{\text{i.i.d}}{\sim}\mathcal{W}_{s}(\alpha)$, $1\le \alpha\le 2$. Then, we have for $p\ge 1$
			\begin{align}\label{Eq_Theo1.3_zong}
			\Big\Vert \sup_{A\in \mathcal{A}}\vert \zeta^\top A^\top A\tilde{\zeta}\vert\Big\Vert_{L_{p}}\lesssim_{\alpha}&\sup_{A\in \mathcal{A}}\Vert \zeta^\top A^\top A\tilde{\zeta}\Vert_{L_{p}}\nonumber\\
			&+\Big\Vert\sup_{A\in\mathcal{A}} \Vert A\tilde{\zeta}\Vert_{2}\Big\Vert_{L_{p}}\big(\gamma_{2}(\mathcal{A}, \Vert\cdot\Vert_{l_{2}\to l_{2}})+\gamma_{\alpha}(\mathcal{A}, \Vert\cdot\Vert_{l_{2}\to l_{\alpha^{*}}}) \big),
		\end{align}
	where $\zeta=(\zeta_{1},\cdots, \zeta_{n})^\top$ and $\tilde{\zeta}$ is an independent copy of $\zeta$. 
\end{mylem}
\begin{proof}
	We shall use Talagrand's chaining argument to bound the decoupled chaos. For convienence, we denote by $d_{\alpha^{*}}$ the distance on $\mathcal{A}$ induced by $\Vert\cdot\Vert_{l_{2}\to l_{\alpha^{*}}}$. Select two admissible sequences of partitions $\mathcal{T}^{(1)}=(\mathcal{T}^{(1)}_{n})_{n\ge 0}$ and $\mathcal{T}^{(2)}=(\mathcal{T}^{(2)}_{n})_{n\ge 0}$ of $\mathcal{A}$ such that
	\begin{align}
		\sup_{A\in \mathcal{A}}\sum_{n\ge 0}2^{n/2}\Delta_{d_{2}}(T_{n}^{(1)}(A))\le 2\gamma_{2}^{\prime}(\mathcal{A}, d_{2})\nonumber
	\end{align}
	and 
	\begin{align}
		\sup_{A\in \mathcal{A}}\sum_{n\ge 0}2^{n/\alpha}\Delta_{d_{\alpha^{*}}}(T_{n}^{(2)}(A))\le 2\gamma_{\alpha}^{\prime}(\mathcal{A}, d_{\alpha^{*}}).\nonumber
	\end{align}
	Let $\mathcal{T}_{0}=\{\mathcal{A}\}$ and 
	\begin{align}
		\mathcal{T}_{n}=\{T^{(1)}\cap T^{(2)}:  T^{(1)}\in \mathcal{T}^{(1)}_{n-1}, T^{(2)}\in \mathcal{T}^{(2)}_{n-1}   \},\quad n\ge 1\nonumber.
	\end{align}
	Then $\mathcal{T}=(\mathcal{T}_{n})_{n\ge 0}$ is increasing and 
	\begin{align}
		\vert \mathcal{T}_{n}\vert\le \vert \mathcal{T}^{(1)}_{n-1}\vert\vert \mathcal{T}^{(2)}_{n-1}\vert\le 2^{2^{n-1}}2^{2^{n-1}}=2^{2^{n}}.\nonumber
	\end{align}
	We define a subset $\mathcal{A}_{n}$ of $\mathcal{A}$ by selecting exactly one point from each $T\in \mathcal{T}_{n}$. In this way, we obtain an admissible sequence $(\mathcal{A}_{n})_{n\ge 0}$ of subsets of $\mathcal{A}$. Let $\pi=\{ \pi_{r}, r\ge 0\}$ be a sequence of functions $\pi_{n}: \mathcal{A}\to \mathcal{A}_{n}$ such that $\pi_{n}(A)=\mathcal{A}_{n}\cap T_{n}(A)$, where $T_{n}(A)$ is the element of $\mathcal{T}_{n}$ containing $A$. Let $l$ be the largest integer such that $2^{l}\le p$.
	
	We make the following decomposition
	\begin{align}\label{Eq_decomposition_1}
		\big\Vert \sup_{A\in \mathcal{A}}\vert \zeta^\top A^\top A\tilde{\zeta}\vert\big\Vert_{L_{p}}\le&\Big\Vert \sup_{A\in \mathcal{A}}\big\vert \zeta^\top A^\top A\tilde{\zeta}-\zeta^\top \pi_{l}(A)^\top\pi_{l}(A)\tilde{\zeta}\big\vert\Big\Vert_{L_{p}}\nonumber\\
		&+\Big\Vert \sup_{A\in \mathcal{A}}\big\vert \zeta^\top \pi_{l}(A)^\top\pi_{l}(A)\tilde{\zeta}\big\vert\Big\Vert_{L_{p}}.
	\end{align}
	
	Fixing $A\in\mathcal{A}$, we have
	\begin{align}\label{Eq_decomposition_2}
		\big\vert \zeta^\top A^\top A\tilde{\zeta}-\zeta^\top \pi_{l}(A)^\top\pi_{l}(A)\tilde{\zeta}\big\vert\le& \sum_{r\ge l}\big\vert\zeta^\top \Lambda_{r+1}(A)^\top\pi_{r+1}(A)\tilde{\zeta}\big\vert\nonumber\\
		&+\sum_{r\ge l}\big\vert\zeta^\top \pi_{r}(A)^\top\Lambda_{r+1}(A)\tilde{\zeta}\big\vert,
	\end{align}
	where $\Lambda_{r+1}A=\pi_{r+1}(A)-\pi_{r}(A)$.
	
	Conditionally on $\tilde{\zeta}$, Lemma \ref{Lem_alpha_2} yields that for $p\ge 2$
	\begin{align}
		\big\Vert \zeta^\top \Lambda_{r+1}(A)^\top\pi_{r+1}(A)\tilde{\zeta}\big\Vert_{L_{p}}\lesssim_{\alpha}&p^{1/2}\big\Vert\Lambda_{r+1}(A)^\top\pi_{r+1}(A)\tilde{\zeta} \big\Vert_{2}\nonumber\\
		&+p^{1/\alpha}\big\Vert\Lambda_{r+1}(A)^\top\pi_{r+1}(A)\tilde{\zeta} \big\Vert_{\alpha^{*}},\nonumber
	\end{align}
	implying with Lemma \ref{Lem_Moments}
	\begin{align}
		\textsf{P}_{\zeta}\Big\{ \vert \zeta^\top S_{r+1}(A, \tilde{\zeta})\vert\ge C(\alpha)\big(\sqrt{t} \Vert S_{r+1}(A, \tilde{\zeta}) \Vert_{2}+t^{1/\alpha} \Vert S_{r+1}(A, \tilde{\zeta}) \Vert_{\alpha^{*}}\big)\Big\}\le e^{2}e^{-t},\nonumber
	\end{align}
	where $S_{r+1}(A, \tilde{\zeta})=\Lambda_{r+1}(A)^\top\pi_{r+1}(A)\tilde{\zeta}$. Note that
	\begin{align}
		\Vert S_{r+1}(A, \tilde{\zeta}) \Vert_{2}\le \Vert\Lambda_{r+1}(A)\Vert_{l_{2}\to l_{2}}\sup_{A\in\mathcal{A}} \Vert A\tilde{\zeta}\Vert_{2}\nonumber
	\end{align}
	and 
	\begin{align}
		\Vert S_{r+1}(A, \tilde{\zeta}) \Vert_{\alpha^{*}}\le \Vert\Lambda_{r+1}(A)\Vert_{l_{2}\to l_{\alpha^{*}}}\sup_{A\in\mathcal{A}} \Vert A\tilde{\zeta}\Vert_{2}.\nonumber
	\end{align}
	Recall that $\big\vert\{ \pi_{r}(A): A\in\mathcal{A} \}\big\vert\le \vert \mathcal{A}_{r}\vert\le 2^{2^{r}}$. Hence, $$\big\vert\{ \Lambda_{r+1}(A)^\top\pi_{r+1}(A): A\in \mathcal{A} \}\big\vert\le 2^{2^{r+3}}.\nonumber$$
	Let the event $\Omega_{t, p}$ be
	\begin{align}
		\bigcap_{r\ge l}\bigcap_{A\in\mathcal{A}}\Big\{ &\vert\zeta^\top S_{r+1}(A, \tilde{\zeta}) \vert\nonumber\\
		&\le C(\alpha)\sup_{A\in\mathcal{A}} \Vert A\tilde{\zeta}\Vert_{2}\big(\sqrt{t}2^{\frac{r}{2}}\Vert\Lambda_{r+1}(A)\Vert_{l_{2}\to l_{2}}+t^{\frac{1}{\alpha}}2^{\frac{r}{\alpha}}\Vert\Lambda_{r+1}(A)\Vert_{l_{2}\to l_{\alpha^{*}}}\big)\Big\}. \nonumber
	\end{align}
	By a union bound, we have for $t>16$
	\begin{align}
		\textsf{P}_{\zeta}\{ \Omega_{t, p}^{c} \}\le \sum_{r\ge l}2^{2^{r+3}}e^{2}e^{-2^{r}t}\le C_{1}(\alpha)\exp\big( -c(\alpha)pt \big).\nonumber
	\end{align}
	If the event $\Omega_{t, p}$ occurs, then
	\begin{align}
		&\sum_{r\ge l}\big\vert\zeta^\top S_{r+1}(A, \tilde{\zeta}) \big\vert\nonumber\\
		\le& C(\alpha)\sup_{A\in\mathcal{A}} \Vert A\tilde{\zeta}\Vert_{2}\Big(\sum_{r\ge l}\sqrt{t}2^{\frac{r}{2}}\Vert \Lambda_{r+1}A\Vert_{l_{2}\to l_{2}}+\sum_{r\ge l}t^{\frac{1}{\alpha}}2^{\frac{r}{\alpha}}\Vert\Lambda_{r+1}A\Vert_{l_{2}\to l_{\alpha^{*}}}\Big).\nonumber
	\end{align}
	Note that we have $\pi_{r+1}(A), \pi_{r}(A)\in T_{r}(A)\subset T^{(1)}_{r-1}(A)$ and so
	\begin{align}
		\Vert \Lambda_{r+1}A\Vert_{l_{2}\to l_{2}}\le \Delta_{d_{2}}\big(T^{(1)}_{r-1}(A)\big)
	\end{align}
	Hence, by our choice of $\mathcal{T}^{(1)}$,
	\begin{align}
		\sum_{r\ge l}2^{\frac{r}{2}}\Vert \Lambda_{r+1}A\Vert_{l_{2}\to l_{2}}\le \sum_{r\ge l}2^{\frac{r}{2}}\Delta_{d_{2}}\big(T^{(1)}_{r-1}(A)\big)\le 2\sqrt{2}\gamma_{2}^{\prime}(\mathcal{A}, d_{2}).\nonumber
	\end{align}
	Analogously, by our choice of $\mathcal{T}^{(2)}$,
	\begin{align}
		\sum_{r\ge l}2^{\frac{r}{\alpha}}\Vert\Lambda_{r+1}A\Vert_{l_{2}\to l_{\alpha^{*}}}\le 2\cdot 2^{\frac{1}{\alpha}}\gamma_{\alpha}^{\prime}(\mathcal{A}, d_{\alpha^{*}}).\nonumber
	\end{align}
	Thus,
	\begin{align}
		\sup_{A\in \mathcal{A}}\sum_{r\ge l}\vert\zeta^\top S_{r+1}(A, \tilde{\zeta}) \vert\le 4C(\alpha)\sup_{A\in\mathcal{A}}\Vert A\tilde{\zeta}\Vert_{2}\big( \sqrt{t}\gamma_{2}^{\prime}(\mathcal{A}, d_{2})+t^{\frac{1}{\alpha}}    \gamma_{\alpha}^{\prime}(\mathcal{A}, d_{\alpha^{*}})\big).\nonumber
	\end{align}
	As a consequence, we obtain for $t>16$
	\begin{align}
		&\textsf{P}\Big\{ \sup_{A\in \mathcal{A}}\sum_{r\ge l}\vert\zeta^\top S_{r+1}(A, \tilde{\zeta}) \vert> C_{2}(\alpha)\sup_{A\in\mathcal{A}}\Vert A\tilde{\zeta}\Vert_{2}t\big( \gamma_{2}^{\prime}(\mathcal{A}, d_{2})+    \gamma_{\alpha}^{\prime}(\mathcal{A}, d_{\alpha^{*}})\big)   \Big\}\nonumber\\
		\le &C_{1}(\alpha)\exp(-c(\alpha)pt).\nonumber
	\end{align}
	A direct integration (Lemma A.5 in \cite{Dirksen_EJP} provides a detailed calculation) and taking expectation with respect to variables $\{ \tilde{\eta}_{i}, i\le n\}$ yield for $p\ge 1$
	\begin{align}
		\Big\Vert\sup_{A\in \mathcal{A}}\sum_{r\ge l}\vert\zeta^\top S_{r+1}(A, \tilde{\zeta}) \vert\Big\Vert_{L_{p}}\lesssim_{\alpha}\Big\Vert\sup_{A\in\mathcal{A}} \Vert A\tilde{\zeta}\Vert_{2}\Big\Vert_{L_{p}}\big( \gamma_{2}(\mathcal{A}, d_{2})+\gamma_{\alpha}(\mathcal{A}, d_{\alpha^{*}})\big)\nonumber.
	\end{align}
	Following the same line, one can give a similar bound for 
	$$\sup_{A\in \mathcal{A}}\sum_{r\ge l}\vert\zeta^\top\pi_{r}(A)^\top \Lambda_{r+1}(A)\tilde{\zeta} \vert.$$
	By virtue of (\ref{Eq_decomposition_2}), we have for $p\ge 1$
	\begin{align}
		&\Big\Vert \sup_{A\in \mathcal{A}}\vert \zeta^\top A^\top A\tilde{\zeta}-\eta^\top \pi_{l}(A)^\top\pi_{l}(A)\tilde{\zeta}\vert\Big\Vert_{L_{p}}\nonumber\\
		\lesssim_{\alpha}&\Big\Vert\sup_{A\in\mathcal{A}} \Vert A\tilde{\zeta}\Vert_{2}\Big\Vert_{L_{p}}\big( \gamma_{2}(\mathcal{A}, d_{2})+\gamma_{\alpha}(\mathcal{A}, d_{\alpha^{*}})\big).\nonumber
	\end{align}
	
	As for the other part of (\ref{Eq_decomposition_1}), we have for $p\ge 1$
	\begin{align}
		\textsf{E}\sup_{A\in \mathcal{A}}\vert \zeta^\top \pi_{l}(A)^\top\pi_{l}(A)\tilde{\eta}\vert^{p}&\le \sum_{B\in T_{l}}	\textsf{E}\vert \zeta^\top B^\top B\tilde{\zeta}\vert^{p}\nonumber\\
		&\le e^{p}\sup_{A\in \mathcal{A}}\textsf{E}\vert \zeta^\top A^\top A\tilde{\zeta}\vert^{p}.\nonumber
	\end{align}
	Thus
	\begin{align}
		\Big\Vert\sup_{A\in \mathcal{A}}\vert \zeta^\top \pi_{l}(A)^\top\pi_{l}(A)\tilde{\zeta}\vert\Big\Vert_{L_{p}}\le e\sup_{A\in \mathcal{A}}\Vert \zeta^\top A^\top A\tilde{\zeta}\Vert_{L_{p}}.\nonumber
	\end{align}
	Then, we conclude the proof by virtue of (\ref{Eq_decomposition_1}). 
\end{proof}

\begin{proof}[Proof of Corollary \ref{Theo_deviation}]

	Let $\zeta_{1},\cdots,  \zeta_{n}\stackrel{\text{i.i.d}}{\sim}\mathcal{W}_{s}(\alpha)$.
	 Lemmas \ref{Lem_Comparison} and  \ref{Theo_bound_decoupled} yield that the quantity $\textsf{E}\sup_{A\in \mathcal{A}}\vert \eta^\top A^\top A\tilde{\eta}\vert$ is further bounded by
\begin{align}\label{Eq_onepart}
	\sup_{A\in \mathcal{A}}\textsf{E}\vert \zeta^\top A^\top A\tilde{\zeta}\vert+\big(\textsf{E}\sup_{A\in\mathcal{A}} \Vert A\tilde{\zeta}\Vert_{2}\big)\big(\gamma_{2}(\mathcal{A}, \Vert\cdot\Vert_{l_{2}\to l_{2}})+\gamma_{\alpha}(\mathcal{A}, \Vert\cdot\Vert_{l_{2}\to l_{\alpha^{*}}}) \big).
\end{align}
 We have by  Remark \ref{Rem_explanations},
\begin{align}\label{Eq_Theoproof1.2_1}
	\sup_{A\in \mathcal{A}}\textsf{E}\vert \zeta^\top A^\top A\tilde{\zeta}\vert\lesssim_{\alpha} \sup_{A\in \mathcal{A}}\Vert A^\top A\Vert_{F}\le M_{F}(\mathcal{A})M_{l_{2}\to l_{2}}(\mathcal{A}),
\end{align}
where $M_{l_{2}\to l_{2}}(\mathcal{A})=\sup_{A\in \mathcal{A}}\Vert A\Vert_{l_{2}\to l_{2}}$ and $M_{F}(\mathcal{A})=\sup_{A\in \mathcal{A}}\Vert A\Vert_{F}$ as defined before.

Observe that by Proposition \ref{Prop_decoupling} and Lemma \ref{Theo_bound_decoupled}
\begin{align}
	\textsf{E}\sup_{A\in \mathcal{A}}\Vert A\tilde{\zeta}\Vert_{2}^{2}&\le \textsf{E} \sup_{A\in\mathcal{A}}\big\vert \Vert A\tilde{\zeta}\Vert_{2}^{2}-\textsf{E}\Vert A\tilde{\zeta}\Vert_{2}^{2}\big\vert  +\sup_{A\in \mathcal{A}}\textsf{E}\Vert A\tilde{\zeta}\Vert_{2}^{2}\nonumber\\
&	\lesssim_{\alpha}\textsf{E}\sup_{A\in \mathcal{A}}\vert \zeta^\top A^\top A\tilde{\zeta}\vert+M^{2}_{F}(\mathcal{A})\nonumber\\
&\lesssim_{\alpha}\textsf{E}\sup_{A\in \mathcal{A}}\Vert A\tilde{\zeta}\Vert_{2}(\gamma_{2}(\mathcal{A}, \Vert\cdot\Vert_{l_{2}\to l_{2}})+\gamma_{\alpha}(\mathcal{A}, \Vert\cdot\Vert_{l_{2}\to l_{\alpha^{*}}}))+M^{2}_{F}(\mathcal{A}).\nonumber
\end{align}
Therefore, 
\begin{align}
	\textsf{E}\sup_{A\in \mathcal{A}}\Vert A\tilde{\zeta}\Vert_{2}\le \Big\Vert  \sup_{A\in \mathcal{A}}\Vert A\tilde{\zeta}\Vert_{2}\Big\Vert_{L_{2}}\lesssim_{\alpha}\Gamma(\alpha, \mathcal{A})+M_{F}(\mathcal{A}),\nonumber
\end{align}
implying $\textsf{E}\sup_{A\in \mathcal{A}}\vert \eta^\top A^\top A\tilde{\eta}\vert$ is further bounded by
\begin{align}\label{Eq_secondpart}
	M_{F}(\mathcal{A})M_{l_{2}\to l_{2}}(\mathcal{A})
	+\Gamma(\alpha, \mathcal{A})\big(\Gamma(\alpha, \mathcal{A})+M_{F}(\mathcal{A})\big),
\end{align}
where $\Gamma(\alpha, \mathcal{A})=\gamma_{2}(\mathcal{A}, \Vert\cdot\Vert_{l_{2}\to l_{2}})+\gamma_{\alpha}(\mathcal{A}, \Vert\cdot\Vert_{l_{2}\to l_{\alpha^{*}}})$ as defined before.

As for $\textsf{E}\sup_{A\in\mathcal{A}}\Vert A^\top A\eta\Vert_{\alpha^{*}}$, we note that
\begin{align}
	\textsf{E}\sup_{A\in\mathcal{A}}\Vert A^\top A\eta\Vert_{\alpha^{*}}\le M_{l_{2}\to l_{\alpha^{*}}}(\mathcal{A})\textsf{E}\sup_{A\in\mathcal{A}}\Vert  A\eta\Vert_{2}.\nonumber
\end{align}
Hence, we have by Lemma \ref{Lem_Comparison}
\begin{align}
\textsf{E}\sup_{A\in\mathcal{A}}\Vert A^\top A\eta\Vert_{\alpha^{*}}\lesssim_{\alpha} M_{l_{2}\to l_{\alpha^{*}}}(\Gamma(\alpha, \mathcal{A})+M_{F}(\mathcal{A})).\nonumber
\end{align}
Then, we have by Corollary \ref{Cor_logconcave}
\begin{align}
	\Big\Vert \sup_{A\in \mathcal{A}}\big\vert \Vert A\xi\Vert_{2}^{2}-\textsf{E}\Vert A\xi\Vert_{2}^{2} \big\vert\Big\Vert_{L_{p}}\lesssim_{\alpha}&\Gamma(\alpha, \mathcal{A})\Big(\Gamma(\alpha, \mathcal{A})+M_{F}(\mathcal{A})\Big)+p^{2/\alpha}M^{2}_{l_{2}\to l_{2}}(\mathcal{A})\nonumber\\
	&+p^{1/\alpha}M_{l_{2}\to l_{\alpha^{*}}}(\mathcal{A})\big(\Gamma(\alpha, \mathcal{A})+M_{F}(\mathcal{A}) \big)\nonumber\\
	&+\sqrt{p}M_{l_{2}\to l_{2}}(\mathcal{A})\big(\Gamma(\alpha, \mathcal{A})+M_{F}(\mathcal{A}) \big).\nonumber
\end{align}
The deviation inequality follows from  Lemma \ref{Lem_Moments}.
\end{proof}

\subsection{R.I.P. of Partial Random Circulant Matrices}
We adopt the notations of Introduction throughout this subsection. Let $1\le \alpha\le 2$, $\eta=(\eta_{1},\cdots, \eta_{n})\in\mathbb{R}^{n}$ be a standard $\alpha$-subexponential  vector. Let $\Phi \in\mathbb{R}^{m\times n}$ be a partial random circulant matrix generated by $\eta$. In this subsection, we shall prove the R.I.P. of $\Phi$.

Let $\Omega\subset\{1,\cdots, n\}$ be a fixed set with $\vert \Omega\vert=m$ and $P_{\Omega}=R_{\Omega}^\top R_{\Omega}$, where $R_{\Omega}: \mathbb{R}^{n}\to \mathbb{R}^{m}$ is the operator restricting a vector $x\in \mathbb{R}^{n}$ to its entries in $\Omega$.  In particular, $P_{\Omega}: \mathbb{R}^{n}\to \mathbb{R}^{n}$, is a projection operator satisfying for $x\in \mathbb{R}^{n}$
\begin{align}
	(P_{\Omega}x)_{l}=x_{l},\,\, l\in \Omega;\quad (P_{\Omega}x)_{l}=0,\,\, l\notin \Omega.\nonumber
\end{align}

Setting $D_{s, n}=\{ x\in \mathbb{R}^{n}: \Vert x\Vert_{2}\le 1, \Vert x\Vert_{0}\le s \}$ and recalling (\ref{Eqisometry}), the restricted isometry constant of $\Phi$ is 
\begin{align}\label{Eq_RIP_chaos}
	\delta_{s}&=\sup_{x\in D_{s, n}}\Big\vert \frac{1}{m}\Vert \Phi x\Vert_{2}^{2}-\Vert x\Vert_{2}^{2}\Big\vert=\sup_{x\in D_{s, n}}\Big\vert \frac{1}{m}\Vert P_{\Omega}(\eta *x)\Vert_{2}^{2}-\Vert x\Vert_{2}^{2}\Big\vert\nonumber\\
	&=\sup_{x\in D_{s, n}}\Big\vert \frac{1}{m}\Vert P_{\Omega}(x*\eta)\Vert_{2}^{2}-\Vert x\Vert_{2}^{2}\Big\vert=\sup_{x\in D_{s, n}}\Big\vert\Vert V_{x}\eta\Vert_{2}^{2}-\Vert x\Vert_{2}^{2}\Big\vert	\nonumber\\
	&=\sup_{x\in D_{s, n}}\Big\vert\Vert V_{x}\eta\Vert_{2}^{2}-\textsf{E}\Vert V_{x}\eta\Vert_{2}^{2}\Big\vert,
\end{align}
where the last equality we use the fact $\textsf{E}\Vert V_{x}\eta\Vert_{2}^{2}=\Vert x\Vert_{2}^{2}$.

\begin{proof}[Proof of Theorem \ref{Theo_RIP}]
	Set $\mathcal{A}=\{ V_{x}: x\in D_{s, n}\}$. By virtue of Corollary \ref{Theo_deviation}, we only need to control the quantities $M_{F}(\mathcal{A}), M_{l_{2}\to l_{\alpha^{*}}}(\mathcal{A})$ and $ \gamma_{\alpha}(\mathcal{A}, \Vert\cdot\Vert_{l_{2}\to l_{\alpha^{*}}})$. 
	
	Note that the matrices $V_{x}$ consist of shifted copies of $x$ in all their $m$ nonzero rows. Hence, the $l_{2}$-norm of each nonzero row is $m^{-1/2}\Vert x\Vert_{2}$. Then, we have for $x\in D_{s, n}$, $\Vert V_{x}\Vert_{F}=\Vert x\Vert_{2}\le 1$, and thus $M_{F}(\mathcal{A})=1$.
	
	Note that
	\begin{align}
		\Vert V_{x}\Vert_{l_{2}\to l_{\alpha^{*}}}&=\sup_{\Vert y\Vert_{2},\Vert z\Vert_{\alpha}\le 1}\vert y^\top V_{x}z\vert\le \sup_{\Vert y\Vert_{2},\Vert z\Vert_{\alpha}\le 1}\frac{1}{\sqrt{m}}\vert y^\top (x*z)\vert\nonumber\\
		&\le \frac{1}{\sqrt{m}}\Vert x\Vert_{1}\le \sqrt{\frac{s}{m}}\Vert x\Vert_{2}\le \sqrt{\frac{s}{m}},\nonumber
	\end{align}
It follows immediately  from that $M_{l_{2}\to l_{\alpha^{*}}}(\mathcal{A})\le \sqrt{s/m}$.

 Lastly, we use the covering number to bound $ \gamma_{\alpha}(\mathcal{A}, \Vert\cdot\Vert_{l_{2}\to l_{\alpha^{*}}})$, see (\ref{Eq_gamma_covering}). Observe that Krahmer et al \cite{Rauhut_CPAM} proved the following bounds for $u\ge 1/\sqrt{m}$
 \begin{align}
 	\log N(\mathcal{A}, \Vert\cdot\Vert_{l_{2}\to l_{2}}, u)\lesssim\frac{s}{m}(\frac{\log n}{u})^{2}\nonumber
 \end{align}
and for $u\le 1/\sqrt{m}$
\begin{align}
	\log N(\mathcal{A}, \Vert\cdot\Vert_{l_{2}\to l_{2}}, u)\lesssim s\log(\frac{en}{su}).\nonumber
\end{align}
A direct calculation (see Section 5.2 for details) yields that for $1\le\alpha<2$
\begin{align}
	 \gamma_{\alpha}(\mathcal{A}, \Vert\cdot\Vert_{l_{2}\to l_{\alpha^{*}}})\le\gamma_{\alpha}(\mathcal{A}, \Vert\cdot\Vert_{l_{2}\to l_{2}})\lesssim_{\alpha}\frac{s^{1/\alpha}}{\sqrt{m}}\log^{2/\alpha} n\nonumber
\end{align}
and for $\alpha=2$
\begin{align}
	\gamma_{2}(\mathcal{A}, \Vert\cdot\Vert_{l_{2}\to l_{2}})\lesssim_{\alpha}\sqrt{\frac{s}{m}}\log s\log n.\nonumber 
\end{align}

 For any $0<\delta<1$, set 
\begin{align}
	m\ge c_{1}(\alpha, L)\delta^{-2}\max\big\{(s^{2/\alpha}\log^{4/\alpha} n),     (s\log^{2}s\log^{2}n)\big\}.\nonumber
\end{align}
Then, it follows
\begin{align}
	\gamma_{2}(\mathcal{A}, \Vert\cdot\Vert_{l_{2}\to l_{2}}), \gamma_{\alpha}(\mathcal{A}, \Vert\cdot\Vert_{l_{2}\to l_{\alpha^{*}}})\lesssim_{\alpha, L}\delta.\nonumber
\end{align}
Choosing the constant $c_{1}(\alpha, L)$ appropriately leads to  $C(\alpha)L^{2}U_{1}(\alpha)\le\delta/2$, where $C(\alpha)$ is the constant in Corollary \ref{Theo_deviation}. Then, we have
\begin{align}
	\textsf{P}\{ \delta_{s}\ge \delta \}&\le \textsf{P}\{ \delta_{s}\ge C(\alpha, L)U_{1}(\alpha)+\delta/2 \}\nonumber\\ 
	&\le e\exp(-c_{2}(\alpha, L)(m/s)^{\alpha/2}\delta^{\alpha}),\nonumber
\end{align}
which concludes the proof.
\end{proof}

\subsection{R.I.P. of Time-Frequency Structured Random Matrices} Recall the notations appearing in Introduction and  observe that for $x\in \mathbb{C}^{m^{2}}$, $\Psi_{h}x=V_{x}\xi$, where 
\begin{align}
	V_{x}=\frac{1}{\sqrt{m}}\sum_{\lambda\in\{0,\cdots,  m-1\}^{2}}x_{\lambda}\boldsymbol{\pi}(\lambda).\nonumber
\end{align}
Note that $\{\frac{1}{\sqrt{m}}\boldsymbol{\pi}(\lambda): \lambda\in \{0,\cdots, m-1  \}^{2}  \}$ is an orthonormal system in the $m\times m$ complex matrices equipped with the Frobenius norm. Hence, we have
\begin{align}
	\textsf{E}\Vert V_{x}\eta\Vert_{2}^{2}=\Vert V_{x}\Vert^{2}_{F}=\Big\Vert m^{-1/2}\sum_{\lambda\in\{0,\cdots,  m-1\}^{2}}x_{\lambda}\boldsymbol{\pi}(\lambda) \Big\Vert_{F}^{2}=\Vert x\Vert_{2}^{2}.\nonumber
\end{align}
Let $D_{s, n}=\{x\in\mathbb{C}^{m^{2}}: \Vert x\Vert_{2}\le 1, \Vert x\Vert_{0}\le s  \}$ and $\mathcal{A}=\{ V_{x}: x\in D_{s, n} \}$, then the restricted isometry constant is 
\begin{align}\label{Eq_chaos_RIP_2}
	\delta_{s}=\sup_{x\in D_{s, n}}\Big\vert \Vert \Psi_{h}x\Vert_{2}^{2}-\Vert x\Vert_{2}^{2}   \Big\vert=\sup_{V_{x}\in \mathcal{A}}\Big\vert \Vert V_{x}\eta\Vert_{2}^{2}-\textsf{E}\Vert V_{x}\eta\Vert_{2}^{2}   \Big\vert.
\end{align}
Hence, we can use Corollary \ref{Theo_deviation} to estimate the tail bound for $\delta_{s}$ again. Below, we shall only sketch the main ideas since the proof is similar to that of Theorem \ref{Theo_RIP}.

\begin{proof}[Proof of Theorem \ref{Theo_time_frequency}]
	First, as noticed above, $\Vert V_{x}\Vert_{F}^{2}=\Vert x\Vert_{2}^{2}\le 1$ for $x\in D_{s, n}$, and thus $M_{F}(\mathcal{A})=1$. Note also that $\Vert \boldsymbol{\pi}(\lambda)\Vert_{F}=1$, so one has for $x\in D_{s, n}$
	\begin{align}
		\Vert V_{x}\Vert_{l_{2}\to l_{\alpha^{*}}}&\le \frac{1}{\sqrt{m}}\sum_{\lambda\in\{0,\cdots,  m-1\}^{2}}\vert x_{\lambda}\vert\cdot \Vert\boldsymbol{\pi}(\lambda)\Vert_{l_{2}\to l_{\alpha^{*}}}\nonumber\\
		&\le \Vert x\Vert_{1}/\sqrt{m}\le \sqrt{s/m}\Vert x\Vert_{2},\nonumber
	\end{align}
which easily implies $M_{l_{2}\to l_{\alpha^{*}}}(\mathcal{A})\le \sqrt{s/m}$.

To bound the Dudley integral, we introduce the following results proved by Krahmer et al \cite{Rauhut_CPAM} (e.g.,  Lemma 5.2)
\begin{align}\label{Eq_Time_RIP_Integral1}
	\log N(\mathcal{A}, \Vert\cdot\Vert_{l_{2}\to l_{2}}, u)\le cs\big(\log(em^{2}/s)+\log(3\sqrt{s/m}/u) \big)
\end{align}
and
\begin{align}\label{Eq_Time_RIP_Integral2}
	\log N(\mathcal{A}, \Vert\cdot\Vert_{l_{2}\to l_{2}}, u)\le c\frac{s}{m}\big(\frac{\log m}{u}\big)^{2},
\end{align}
where $0<u\le \sqrt{s/m}$ and $c$ is an absolute constant.

On the other hand, by virtue of the fact $\alpha^{*}\ge 2$ and $M_{l_{2}\to l_{\alpha^{*}}}(\mathcal{A})\le \sqrt{s/m}$, it follows
\begin{align}\label{Eq_3.22}
	\gamma_{\alpha}(\mathcal{A}, \Vert\cdot\Vert_{l_{2}\to l_{\alpha^{*}}})\le& \gamma_{\alpha}(\mathcal{A}, \Vert\cdot\Vert_{l_{2}\to l_{2}})
	\le \int_{0}^{\sqrt{s/m}}\log^{1/\alpha} N(\mathcal{A}, \Vert\cdot\Vert_{l_{2}\to l_{2}}, u)\, du\nonumber\\
	=& \int_{0}^{\sqrt{1/m}}+\int_{\sqrt{1/m}}^{\sqrt{s/m}}\log^{1/\alpha} N(\mathcal{A}, \Vert\cdot\Vert_{l_{2}\to l_{2}}, u)\, du.
\end{align}
As for the case $1\le \alpha<2$, applying (\ref{Eq_Time_RIP_Integral1})  and  (\ref{Eq_Time_RIP_Integral2}) to the first and the second integral of (\ref{Eq_3.22}) respectively, we have
\begin{align}
	\gamma_{\alpha}(\mathcal{A}, \Vert\cdot\Vert_{l_{2}\to l_{\alpha^{*}}})\lesssim_{\alpha}\frac{s^{1/\alpha}}{\sqrt{m}}\log^{2/\alpha} n.\nonumber
\end{align}
Turn to the case $\alpha=2$. We have by a similar argument
\begin{align}
	\gamma_{2}(\mathcal{A}, \Vert\cdot\Vert_{l_{2}\to l_{2}})\lesssim_{\alpha}\sqrt{\frac{s}{m}}\log s\log n.\nonumber 
\end{align}
At last, we conclude the desired result due to (\ref{Eq_remark1.6})
\end{proof}

\section{Discussions}
(1) In this section, we shall show an improved result of Corollary \ref{Theo_deviation}. To this end, we first introduce the following lemma, a comparison of weak and strong moments. One can regard this lemma as a particular case of a general result obtained by Latała \cite{Latala_Studia_Math}. For the convenience of reading, we shall briefly show how to derive Lemma \ref{Lem_Latala_EJP} from Latała’s result. We also refer to interested readers to \cite{Latala_mathematika,Latala_EJP} for some similar results in general cases.

 \begin{mylem}\label{Lem_Latala_EJP}
	Let $\{\xi_{i}, i\ge 1\}$ be a sequence of independent symmetric variables with log-concave tails. 
	Then for any nonempty $T\subset \{t :\sum t_{i}^{2}\le 1\}$, we have for $p\ge 1$
	\begin{align}
		\Big\Vert  \sup_{t\in T} \big\vert \sum t_{i}\xi_{i}\big\vert\Big\Vert_{L_{p}}\lesssim \textsf{E}\sup_{t\in T} \big\vert \sum t_{i}\xi_{i}\big\vert+\sup_{t\in T} \Big\Vert\big\vert \sum t_{i}\xi_{i}\big\vert\Big\Vert_{L_{p}}.\nonumber
	\end{align}
\end{mylem} 
\begin{proof}
	Set
	\begin{align}
		U_{i}(t)=-\log \textsf{P}\big\{ \vert\xi_{i} \vert\ge t \big\},\quad t\ge 0.\nonumber
	\end{align}
	Without loss of generality, we assume that $U_{i}(x)=1$. Define 
	$$
	\hat{U_{i}}(t)=
	\begin{cases}
		t^{2},\quad &\vert t\vert\le 1;\\
		U_{i}(\vert t\vert),\quad &\vert t\vert\ge 1.
	\end{cases}
	$$
	For sequences $(a_{i})$ of real numbers and $(v_{i})$ of vectors in some Banach space $F$, we define for $u>0$
	\begin{align}
		\big\Vert(a_{i})\big\Vert_{\mathcal{U}, u}=\sup\big\{ \sum a_{i}b_{i}:\,\, \sum\hat{U_{i}}(b_{i})\le u\big\}\nonumber
	\end{align}
	and 
	\begin{align}
		\big\Vert(v_{i})\big\Vert^{w}_{\mathcal{U}, u}=\sup\Big\{\big\Vert\big( v^{*}(v_{i}) \big) \big\Vert_{\mathcal{U}, u}:  v^{*}\in F^{*}, \Vert v^{*}\Vert^{*}\le 1 \Big\},\nonumber
	\end{align}
	where $F^{*}$ is the dual space of $F$ with the dual norm $\Vert\cdot\Vert^{*}$.  Let $v_{i}$ be vectors of some Banach space $F$ such that the series $\sum v_{i}\xi_{i}$ is almost surely convergent. Then, Theorem 1 in \cite{Latala_Studia_Math} yields for $p\ge 1$
	\begin{align}\label{Eq_Section4_Latala_studia}
		\big(\textsf{E}\Vert \sum v_{i}\xi_{i}\Vert^{p}\big)^{1/p}\asymp \textsf{E}\Vert \sum v_{i}\xi_{i} \Vert+\big\Vert(v_{i})\big\Vert^{w}_{\mathcal{U}, p},
	\end{align}
	where $\Vert\cdot\Vert$ is the norm on $F$.
	
	To obtain Lemma \ref{Lem_Latala_EJP}, we first define the following norm on $l_{2}$:
	\begin{align}
		\Vert x\Vert=\sup_{t\in T}\big\vert \sum_{i} x_{i}t_{i}\big\vert.\nonumber
	\end{align}
	where $T$ is a nonempty set satisfying $\text{span}(T)=l_{2}:=\{t: \sum t_{i}^{2}<\infty \}$. Let $e_{i}, i\ge 1$ be vectors in $l_{2}$ such that the $i$-th coordinate is $1$ and the other coordinates are $0$. Then (\ref{Eq_Section4_Latala_studia}) yields for $\sum e_{i}\xi_{i}$
	\begin{align}\label{Eq_Section_4.2}
		\Big(\textsf{E}\sup_{t\in T}\big\vert \sum t_{i}\xi_{i}\big\vert^{p}\Big)^{1/p}\asymp \textsf{E}\sup_{t\in T}\big\vert \sum t_{i}\xi_{i} \big\vert+\sup_{t\in T}\big\Vert(t_{i})\big\Vert_{\mathcal{U}, p}.
	\end{align}
	We use (\ref{Eq_Section4_Latala_studia}) again in the real case,
	\begin{align}
		\Big(\textsf{E}\big\vert \sum t_{i}\xi_{i}\big\vert^{p}\Big)^{1/p}\asymp \textsf{E}\big\vert \sum t_{i}\xi_{i} \big\vert+\big\Vert(t_{i})\big\Vert_{\mathcal{U}, p}.\nonumber
	\end{align}
	Then, 
	\begin{align}\label{Eq_Section_4.3}
		\big\Vert(t_{i})\big\Vert_{\mathcal{U}, p}\lesssim \Big(\textsf{E}\big\vert \sum t_{i}\xi_{i}\big\vert^{p}\Big)^{1/p}.
	\end{align}
	Combining (\ref{Eq_Section_4.2}) and (\ref{Eq_Section_4.3}), we finish the proof of Lemma \ref{Lem_Latala_EJP} for some special $T$. As for a general $T$, we let $T_{\delta}=T\cup \delta B_{2}^{\infty}$. Note that, $\text{span}(T_{\delta})=l_{2}$. Then, go with $\delta\to 0$ and obtain the desired result.
\end{proof}

\begin{mycorollary1.2}\label{1}
	In the setting of Corollary \ref{Theo_deviation}, we have for $p\ge 1$
	\begin{align}
		\Big\Vert \sup_{A\in\mathcal{A}}\big\vert \Vert A\xi\Vert_{2}^{2}-\textsf{E}\Vert A\xi\Vert_{2}^{2} \big\vert\Big\Vert_{L_{p}}\lesssim_{L, \alpha}& U_{1}(\alpha)+\sqrt{p}U_{2}^{\prime}(\alpha)+p^{1/\alpha}U_{3}^{\prime}(\alpha)\nonumber\\
		&+p^{2/\alpha}M^{2}_{l_{2}\to l_{2}}(\mathcal{A}).\nonumber
	\end{align}
Here, $U_{2}^{\prime}(\alpha)=M_{l_{2}\to l_{2}}(\mathcal{A})\Gamma(\alpha, \mathcal{A})+\sup_{A\in}\Vert A^\top A\Vert_{F}$ and  $U_{3}^{\prime}(\alpha)=M_{l_{2}\to l_{\alpha^{*}}}(\mathcal{A})\Gamma(\alpha, \mathcal{A})$.
\end{mycorollary1.2}

\begin{proof}[Proof of Corollary \ref{Theo_deviation}]
	Let $\zeta=(\zeta_{1}, \cdots, \zeta_{n})$ and $\zeta_{1}, \cdots, \zeta_{n}\stackrel{i.i.d.}{\sim}\mathcal{W}_{s}(\alpha)$. Set $S=\{ A^\top x: x\in B^{n}_{2}, A\in\mathcal{A} \}$. We have by Lemma \ref{Lem_Latala_EJP} 
	\begin{align}\label{Eq_Lpbound}
		\Big\Vert\sup_{A\in\mathcal{A}} \Vert A\zeta\Vert_{2}\Big\Vert_{L_{p}}&=\Big(\textsf{E}\sup_{A\in\mathcal{A}, x\in B_{2}^{n}}\vert x^\top A\zeta\vert^{p}   \Big)^{1/p}=\Big(\textsf{E}\sup_{u\in S}\vert u^\top \zeta\vert^{p}   \Big)^{1/p}\nonumber\\
		&\lesssim_{\alpha} \textsf{E}\sup_{u\in S}\vert u^\top \zeta\vert+\sup_{u\in S}\Vert u^\top \zeta\Vert_{L_{p}}\nonumber\\
		&\lesssim_{\alpha} \textsf{E}\sup_{u\in S}\vert u^\top \zeta\vert+\sqrt{p}\sup_{u\in S}\Vert u\Vert_{2}+p^{1/\alpha}\sup_{u\in S}\Vert u\Vert_{\alpha^{*}}\nonumber\\
		&=\textsf{E}\sup_{A\in \mathcal{A}}\Vert A\zeta\Vert_{2}+\sqrt{p}M_{l_{2}\to l_{2}}(\mathcal{A})+p^{1/\alpha}M_{l_{2}\to l_{\alpha^{*}}}(\mathcal{A}).
	\end{align}
Note that
\begin{align}
	\textsf{E}\sup_{A\in\mathcal{A}}\Vert A\zeta\Vert_{2}\lesssim_{\alpha}\Gamma(\alpha, \mathcal{A})+M_{F}(\mathcal{A}).\nonumber
\end{align}
Then, we conclude the proof by Proposition \ref{Prop_decoupling} and Lemma \ref{Theo_bound_decoupled}
\end{proof}

(2) Let $\varepsilon=(\varepsilon_{1}, \cdots,\varepsilon_{n})$, where $\varepsilon_{1}, \cdots, \varepsilon_{n}$ are i.i.d. symmetric Bernoulli variables. Readers may notice that the structured random matrices considered in \cite{Rauhut_CPAM} are generated by $\varepsilon$ but Krahmer et al \cite{Rauhut_CPAM}  proved the R.I.P. of the structured random matrices generated by a standard sub-Gaussian vector. Obviously, the optimality of the final result will be affected by this approach.

Bednorz and Latała \cite{Bednorz_annals} presented a positive solution to the so-called Bernoulli Conjecture, which was open for about 25 years. In particular, let $T$ be a nonempty set of $l_{2}$. They showed
\begin{align}
	\sup_{t\in T}\sum_{i=1}^{\infty}t_{i}\varepsilon_{i}\asymp\inf\{\sup_{t\in T_{1}}\Vert t\Vert_{1}+\gamma_{2}(T_{2}, \Vert\cdot\Vert_{2}): T\subset T_{1}+T_{2}  \}.\nonumber
\end{align}
This celebrated result makes it possible to show the R.I.P. of the structured random matrices induced by $\varepsilon$ directly, which may improve the results in \cite{Rauhut_CPAM}. However, there are still many obstacles to solving this problem. For example, how to determine the decomposition of $T$ in a specific case.

\section{Appendix}
\subsection{R.I.P. of $\alpha$-Subexponential Random Matrices}
We know that subgaussian random matrices (i.e., random matrices with independent standard subgaussian entries) satisfy R.I.P.
In this subsection, we would like to add a proof that $\alpha$-subexponential random matrices also satisfy R.I.P., where $1\le \alpha\le 2$. 
\begin{mypro}
	 Let $\Phi=1/\sqrt{m}(\xi_{ij})_{m\times n}$ be a random matrix such that for $1\le i\le m, 1\le j\le n$
	\begin{align}
		\textsf{E}\xi_{ij}=0,\quad \textsf{E}\xi_{ij}^{2}=1,\quad \max_{ij}\Vert\xi_{ij} \Vert_{\Psi_{\alpha}}\le L.\nonumber
	\end{align}
	Then for $\delta, \varepsilon\in (0, 1)$, $\delta_{s}\le \delta$ with probability at least $1-\exp\big(-c(\alpha, L)s\log(en/s)\big)$ provided
	\begin{align}
		m\ge C(\alpha, L)\delta^{-2}\big(s\log(en/s) \big)^{2/\alpha}. \nonumber
	\end{align}
\end{mypro}

\begin{proof}
	Fix $x=(x_{1}, \cdots, x_{n})^{\top}\in \mathbb{R}^{n}$ and let $\sqrt{m}V_{x}$ be an $m\times nm$ block-diagonal matrix as follows
	\begin{align}
		\begin{pmatrix}
			x^{\top}& 0 &\cdots & 0 \\
			0& x^{\top} &\cdots & 0 \\
			\vdots & \vdots & \ddots & \vdots\\
			0 & 0 & \cdots & x^{\top}  \nonumber
		\end{pmatrix}.
	\end{align}
	Hence, we have $\Phi x=V_{x}\xi$, where $\xi=(\xi_{11},\cdots,\xi_{1n}, \xi_{21}, \cdots,\xi_{mn})^{\top}$, a vector of length $nm$.
	
	Let $\mathcal{A}:=\{ V_{x}: \Vert x\Vert_{0}\le s, \Vert x\Vert_{2}\le 1  \}$. It is obvious that
	\begin{align}
		\sup_{V_{x}\in\mathcal{A}}\Vert V_{x}\Vert_{F}=\sup_{x\in D_{s, n}}\Vert x\Vert_{2}=1,\nonumber
	\end{align}
	where $D_{s, n}=\{x\in\mathbb{R}^{n}: \Vert x\Vert_{0}\le s, \Vert x\Vert_{2}\le 1  \}$.
	Note that the  operator norm of a block-diagonal matrix is the maximum of the operator norms of the diagonal blocks. Then for $1\le \alpha\le 2$
	\begin{align}
		\sup_{V_{x}\in\mathcal{A}}\Vert V_{x} \Vert_{l_{2}\to l_{\alpha^{*}}}=\sup_{V_{x}\in\mathcal{A}}\Vert V_{x} \Vert_{l_{2}\to l_{2}}=\frac{1}{\sqrt{m}} 	\sup_{x\in D_{s, n}}\Vert x \Vert_{2}=\frac{1}{\sqrt{m}},\nonumber
	\end{align}
	where $\alpha^{*}=\alpha/(\alpha-1)$.
	Observe that 
	\begin{align}\label{Eq_appendix_gamma_iid}
		\gamma_{\alpha}\big(\mathcal{A}, \Vert \cdot\Vert_{l_{2}\to l_{\alpha^{*}}}  \big)&\le \gamma_{\alpha}\big(D_{s, n}, \Vert \cdot\Vert_{2}/\sqrt{m}  \big)\nonumber\\
		&\lesssim \frac{1}{\sqrt{m}}\int_{0}^{1}\log^{1/\alpha}N(D_{s, n}, \Vert\cdot\Vert_{2}, u)\, du.
	\end{align}
	At the same time, the volumetric argument yields
	\begin{align}
		N(D_{s, n}, \Vert\cdot\Vert_{2}, u)\le (\frac{en}{s})^{s}(1+\frac{2}{u})^{s}.\nonumber
	\end{align}
	Hence, (\ref{Eq_appendix_gamma_iid}) is further bounded by
	\begin{align}
		C\frac{s^{1/\alpha}}{\sqrt{m}}\Big( \log^{1/\alpha}(en/s)+\int_{0}^{1}\log^{1/\alpha}(1+2/u)\, du   \Big)\le C_{1}\frac{1}{\sqrt{m}}\big(s\log(en/s)\big)^{1/\alpha}.\nonumber
	\end{align}
	Theorem \ref{Theo_deviation} implies the desired assertion.
\end{proof}

\subsection{Bounds of $\gamma_{\alpha}$-functionals}
This subsection is devoted to a detailed calculation for the estimates of $\gamma_{\alpha}$-functionals. Let us start with the case $\alpha=2$. First, observe that
\begin{align}
	\gamma_{2}(\mathcal{A}, \Vert\cdot\Vert_{l_{2}\to l_{2}})&\lesssim \int_{0}^{M_{l_{2}\to l_{2}}(\mathcal{A})} \log^{\frac{1}{2}}N(\mathcal{A}, \Vert\cdot\Vert_{l_{2}\to l_{2}}, u)\, du\nonumber\\
	&\lesssim \int_{0}^{\frac{1}{\sqrt{m}}} s^{\frac{1}{2}}\log^{\frac{1}{2}}(\frac{en}{su})\, du+\int_{\frac{1}{\sqrt{m}}}^{\sqrt{\frac{s}{m}}} \frac{\sqrt{s}}{\sqrt{m}u}\log n\, du.\nonumber
\end{align}
Also, a simple calculus yields
\begin{align}
	\int_{0}^{\frac{1}{\sqrt{m}}} s^{\frac{1}{2}}\log^{\frac{1}{2}}(\frac{en}{su})\, du&\le \sqrt{\frac{s}{m}}\log^{\frac{1}{2}}(\frac{en}{s})+\sqrt{s}\int_{0}^{\frac{1}{\sqrt{m}}}\log^{\frac{1}{2}}\frac{1}{u}\, du\nonumber\\
	&\le \sqrt{\frac{s}{m}}\log^{\frac{1}{2}}(\frac{en}{s})+\sqrt{s}\int_{\log\sqrt{m}}^{\infty}\sqrt{t}e^{-t}\, dt\nonumber\\
	&\le \sqrt{\frac{s}{m}}\log^{\frac{1}{2}}(\frac{en}{s})+\sqrt{\frac{s}{m}}\int_{\log\sqrt{m}}^{\infty}\sqrt{t}e^{-t+\log\sqrt{m}}\, dt\nonumber\\
	&\lesssim \sqrt{\frac{s}{m}}\sqrt{\log m}\nonumber
\end{align}
and
\begin{align}
     \int_{\frac{1}{\sqrt{m}}}^{\sqrt{\frac{s}{m}}} \frac{\sqrt{s}}{\sqrt{m}u}\log n\, du\lesssim \sqrt{\frac{s}{m}}\log n\log s.\nonumber
\end{align}
Hence, we have 
\begin{align}
	\gamma_{2}(\mathcal{A}, \Vert\cdot\Vert_{l_{2}\to l_{2}})\lesssim \sqrt{\frac{s}{m}}\log n\log s.\nonumber
\end{align}
 
 As for the case $1\le \alpha<2$, we have by virtue of $\alpha^{*}>2$
 \begin{align}
 	\gamma_{\alpha}(\mathcal{A}, \Vert\cdot\Vert_{l_{2}\to l_{\alpha^{*}}})\le \gamma_{\alpha}(\mathcal{A}, \Vert\cdot\Vert_{l_{2}\to l_{2}})\lesssim_{\alpha}& \int_{0}^{\frac{1}{\sqrt{m}}} s^{\frac{1}{\alpha}}\log^{\frac{1}{\alpha}}(\frac{en}{su})\, du\nonumber\\
 	&+\int_{\frac{1}{\sqrt{m}}}^{\sqrt{\frac{s}{m}}} \Big(\frac{s}{mu^{2}}\Big)^{\frac{1}{\alpha}}\log^{\frac{2}{\alpha}} n\, du.\nonumber
 \end{align}
Similarly, we have
\begin{align}
	\int_{0}^{\frac{1}{\sqrt{m}}} s^{\frac{1}{\alpha}}\log^{\frac{1}{\alpha}}(\frac{en}{su})\, du\lesssim_{\alpha} \frac{s^{\frac{1}{\alpha}}}{\sqrt{m}}\log^{\frac{1}{\alpha}} m\nonumber
\end{align}
and
\begin{align}
	\int_{\frac{1}{\sqrt{m}}}^{\sqrt{\frac{s}{m}}} \Big(\frac{s}{mu^{2}}\Big)^{\frac{1}{\alpha}}\log^{\frac{2}{\alpha}} n\, du\lesssim_{\alpha} \frac{s^{\frac{1}{\alpha}}}{\sqrt{m}}\log^{\frac{2}{\alpha}} n.\nonumber 
\end{align}
In combination, it follows
\begin{align}
	\gamma_{\alpha}(\mathcal{A}, \Vert\cdot\Vert_{l_{2}\to l_{\alpha^{*}}})\lesssim \frac{s^{\frac{1}{\alpha}}}{\sqrt{m}}\log^{\frac{2}{\alpha}} n,\nonumber
\end{align}
as desired.

%%%%%%%%%%%%%%%%%%%%%%%%s%%%%%2^{}eme
%% Example wit.h single Appendix:            %%
%%%%%%%%%%%%%%%%%%%%%%%%%%%% %%%%%%%%%%%%%%%%%

%%%%%%%%%%%%%%%%%%%%%%%%%%%%%%%%%%%%%%%%%%%%%%
%% Support information, if any,             %%
%% should be provided in the                %%ow
%% Acknowledgements section.                %%
%%%%%%%%%%%%%%%%%%%%%%%%%%%%%%%%%%%%%%%%%%%%%%

\textbf{Acknowledgment} The authors are grateful to Rafał Latała for his fruitful discussions regarding Lemma \ref{Lem_Latala_EJP}. Su was partly supported by NSFC (Grant Nos. 12271475 and U23A2064). Wang was partly supported by NSFC (Grant Nos. 12071257, 12371148);   National Key R$\&$D Program of China (No.2018YFA0703900); Shandong Provincial Natural Science Foundation (No. ZR2019ZD41).


\end{document}